\documentclass[oneside,english]{amsart}
\usepackage[T1]{fontenc}
\usepackage[latin9]{inputenc}
\usepackage{calc}
\usepackage{amstext}
\usepackage{amsthm}
\usepackage{amssymb}
\usepackage{wasysym}
\usepackage{xargs}[2008/03/08]

\makeatletter
\numberwithin{equation}{section}
\numberwithin{figure}{section}
\theoremstyle{plain}
\newtheorem{thm}{\protect\theoremname}[section]
\theoremstyle{definition}
\newtheorem{defn}[thm]{\protect\definitionname}
\theoremstyle{remark}
\newtheorem{rem}[thm]{\protect\remarkname}
\theoremstyle{plain}
\newtheorem{lem}[thm]{\protect\lemmaname}
\theoremstyle{plain}
\newtheorem{prop}[thm]{\protect\propositionname}
\theoremstyle{definition}
\newtheorem{example}[thm]{\protect\examplename}
\theoremstyle{definition}
\newtheorem{problem}[thm]{\protect\problemname}
\theoremstyle{plain}
\newtheorem{cor}[thm]{\protect\corollaryname}
\theoremstyle{plain}
\newtheorem{fact}[thm]{\protect\factname}

\usepackage{pdflscape}

\usepackage{tikz}

\usepackage{pgfplots}

\usetikzlibrary{arrows}

\usetikzlibrary{automata}

\usetikzlibrary{chains}

\usetikzlibrary{positioning}

\usetikzlibrary{backgrounds}

\usetikzlibrary{calc,decorations.pathreplacing}

\usepackage{amssymb,amsmath}

\pgfplotsset{compat=1.9}

\makeatletter 
\g@addto@macro\@floatboxreset{\centering} 
\makeatother

\usepackage{enumitem}
\setenumerate{ label= (\alph*)}
\setenumerate[2]{ label= (\alph{enumi}\arabic*)} 

\makeatother

\usepackage{babel}
\providecommand{\corollaryname}{Corollary}
\providecommand{\definitionname}{Definition}
\providecommand{\examplename}{Example}
\providecommand{\factname}{Fact}
\providecommand{\lemmaname}{Lemma}
\providecommand{\problemname}{Problem}
\providecommand{\propositionname}{Proposition}
\providecommand{\remarkname}{Remark}
\providecommand{\theoremname}{Theorem}

\begin{document}
\title{Universal co-Extensions of torsion abelian groups}
\author{Alejandro Argud\'in-Monroy}
\email{argudin@ciencias.unam.mx}
\address{Centro de Ciencias Matem\'aticas, Campus Morelia, Universidad Nacional Aut\'onoma de M\'exico, Antigua Carretera a P\'atzcuaro 8701, Colonia Ex Hacienda San Jos\'e de la Huerta, Morelia, Michoac\'an, M\'exico, C.P. 58089}
\thanks{The first named author was supported by a postdoctoral fellowship from DGAPA-UNAM}
\author{Carlos E. Parra}
\email{carlos.parra@uach.cl}
\address{Instituto de Ciencias F\'\i sicas y Matem\'aticas, Edificio Emilio Pugin, Campus Isla Teja, Universidad Austral de Chile, 5090000 Valdivia, CHILE}
\thanks{The second named author was supported by CONICYT/FONDECYT/REGULAR/1200090}
\subjclass[2000]{18g15, 18e10, 20k10, 20k25, 20k35, 20k40}

\maketitle
\newcommandx\Mod[1][usedefault, addprefix=\global, 1=R]{\operatorname{Mod}\left(#1\right)}%

\newcommandx\modd[1][usedefault, addprefix=\global, 1=R]{\operatorname{mod}\left(#1\right)}%

\newcommandx\Ker[1][usedefault, addprefix=\global, 1=M]{\operatorname{Ker}\left(#1\right)}%

\newcommandx\Ann[1][usedefault, addprefix=\global, 1=M]{\operatorname{Ann}#1}%

\newcommandx\im[1][usedefault, addprefix=\global, 1=M]{\operatorname{Im}\left(#1\right)}%

\newcommandx\Cok[1][usedefault, addprefix=\global, 1=M]{\operatorname{Coker}\left(#1\right)}%

\newcommandx\pdr[2][usedefault, addprefix=\global, 1=M, 2=\mathcal{A}]{\mathrm{pd}{}_{#2}\left(#1\right)}%
\newcommandx\idr[2][usedefault, addprefix=\global, 1=M, 2=\mathcal{B}]{\mathrm{id}{}_{#2}\left(#1\right)}%

\newcommandx\smd[1][usedefault, addprefix=\global, 1=\mathcal{M}]{\operatorname{smd}\left(#1\right)}%

\newcommandx\suc[5][usedefault, addprefix=\global, 1=N, 2=M, 3=K, 4=, 5=]{#1\overset{#4}{\hookrightarrow}#2\overset{#5}{\twoheadrightarrow}#3}%

\newcommandx\Ext[4][usedefault, addprefix=\global, 1=M, 2=i, 3=\mathcal{A}, 4=X]{\mathrm{Ext}{}_{#3}^{#2}\left(#1,#4\right)}%

\newcommandx\p[2][usedefault, addprefix=\global, 1=\mathcal{A}, 2=\mathcal{B}]{\left(#1,#2\right)}%

\newcommandx\Hom[3][usedefault, addprefix=\global, 1=\mathcal{A}, 2=M, 3=N]{\mathrm{Hom}{}_{#1}(#2,#3)}%

\newcommandx\End[2][usedefault, addprefix=\global, 1=R, 2=M]{\mathrm{End}{}_{#1}(#2)}%

\newcommandx\proj[1][usedefault, addprefix=\global, 1=R]{\operatorname{proj}\left(#1\right)}%

\newcommandx\Inj[1][usedefault, addprefix=\global, 1=R]{\operatorname{Inj}\left(#1\right)}%

\newcommandx\inj[1][usedefault, addprefix=\global, 1=R]{\operatorname{inj}\left(#1\right)}%

\newcommandx\Proj[1][usedefault, addprefix=\global, 1=R]{\operatorname{Proj}\left(#1\right)}%

\newcommandx\Homk[4][usedefault, addprefix=\global, 1=\mathcal{K}(R), 2=\sigma, 3=\omega, 4=1]{\mathrm{Hom}{}_{#1}(#2,#3[#4])}%

\newcommandx\Add[1][usedefault, addprefix=\global, 1=\mathcal{M}]{\operatorname{Add}\left(#1\right)}%

\newcommandx\add[1][usedefault, addprefix=\global, 1=\mathcal{M}]{\operatorname{add}\left(#1\right)}%

\newcommandx\Gen[1][usedefault, addprefix=\global, 1=M]{\operatorname{Gen}\left(#1\right)}%

\newcommandx\colim[1][usedefault, addprefix=\global, 1=M]{\mathsf{colim}_{\Sigma}\left(#1\right)}%
 
\global\long\def\Fun{\operatorname{Fun}}%
 
\global\long\def\A{\mathcal{A}}%
 
\global\long\def\B{\mathcal{B}}%
 
\global\long\def\col{\mathsf{colim}}%
 
\global\long\def\limite{\mathsf{lim}}%
\global\long\def\Extu{\mathsf{Ext.u.}}%

\global\long\def\C{\mathcal{C}}%
\global\long\def\T{\mathcal{T}}%
\global\long\def\F{\mathcal{F}}%
\global\long\def\Ab{\mathsf{Ab}}%
\global\long\def\G{\mathcal{G}}%

\begin{abstract}
In \cite{parra2021tilting}, a theory of universal extensions in abelian
categories is developed, moreover, the notion of Ext-universal object
is presented. We show that an Ab3 abelian category which is Ext-small,
satisfies the Ab4 condition if, and only if, each object of it is
Ext-universal. In particular, this means that there are torsion abelian
groups that are not co-Ext-universal in the category of torsion abelian
groups. In this sense, we characterize all torsion abelian groups
which are co-Ext-universal in such category. Namely, we show that
such groups are the ones that admit a decomposition $Q\oplus R$,
where $Q$ is injective and $R$ is a reduced group on which each
$p$-component is bounded. 
\end{abstract}

\section{Introduction}

In \cite{Ab}, Grothendieck axiomatizes the desirable properties that
a category must have in order to be studied by means of homological
algebra techniques. Specifically, axioms Ab1 and Ab2 define abelian
categories, Ab3 consists of the existence of arbitrary coproducts, Ab4
the existence and exactness of arbitrary coproducts, and Ab5 the existence
and exactness of direct limits. Additionally, the dual axioms are
denoted as $\mathrm{Ab}n^{\ast}$. Nowadays, an abelian category is said to be a \emph{Grothendieck
category} if it is Ab5 and admits a generator. 

Given their importance, the $\mathrm{Ab}n$ properties have been studied extensively.
For example, it is known that Grothendieck categories are Ab3{*} and
have enough injectives. Moreover, it is known that these categories rarely are Ab4{*} or have enough projectives. A classical result concerning the characterization of the $\mathrm{Ab}4^{\ast}$ condition in Grothendieck categories is a theorem of Roos (see \cite[Theorem 1]{roos1966produits}
or \cite[Corollary 1.4]{roos2006derived}), which states that a
Grothendieck category $\mathcal{G}$ is Ab4{*} if and only if every
object $X\in\mathcal{G}$ admits a \emph{projective effacement} $P\rightarrow X$,
where a projective effacement is an epimorphism $\alpha$ such that
$\Ext[\alpha][1][\mathcal{G}][G]=0$ for all $G\in\mathcal{G}$ (see
Section \ref{sec:eff}). 

A similar (but dual) result has been proved recently by the second named author  with Saor\'in and
Virili in \cite[Proposition 5.9]{parra2021tilting}. In more detail, if the category is Ab4 and $\operatorname{Ext}$-small,
then every object is \emph{$\operatorname{Ext}^{1}$-universal}. Recall
that an object $B$ in an abelian category $\mathcal{A}$ is $\operatorname{Ext}^{1}$-universal
if, for every $A\in\mathcal{A}$, there is a monomorphism $\alpha:A\rightarrow E$
such that $\Cok[\alpha]\cong B^{(X)}$, for some set $X$, and $\Ext[B][1][\mathcal{G}][\alpha]=0$
(see Definition \ref{def:extuniversal}). In this case, the short
exact sequence $\suc[A][E][B^{(X)}][\alpha]$ is called a \emph{universal
extension} of $B$ by $A$. 

In \cite[Theorem 4.8]{argudin2021yoneda}, the first named author
has shown that the condition Ab4 is equivalent to the natural morphism
\[
\Psi:\operatorname{Ext}_{\mathcal{A}}^{1}(\bigoplus_{i\in I}B_{i},A)\rightarrow\prod_{i\in I}\Ext[B_{i}][1][][A],
\]
 being always bijective, for all object $A$ and $(B_i)_{i\in I}$ family of objects in such category (see Theorem \ref{thm:Ab4 vs ext}). The aim
of this paper is to show from this fact that, for $\operatorname{Ext}$-small
abelian categories, the condition Ab4 is equivalent to every object
being $\operatorname{Ext}^{1}$-universal. Moreover, we will revisit the aforementioned result of 
Roos to construct projective effacements from
\emph{universal co-extensions}, these latter being the dual notion of universal
extensions. Having done so, we turn our attention to abelian
categories that are Ab3 but not Ab4. Specifically, we will study
the class of $\operatorname{Ext}^{1}$-universal objects; our main
goal is to find a characterization for this type of objects.
In such a generality, our task seems to be unapproachable. However,
we give a sufficient condition for an object to be $\operatorname{Ext}^{1}$-universal in Lemma \ref{lem.sufficientcondi}. In view of this result, we address our problem in a specific dual situation. Namely,
we characterize the \emph{co-$\operatorname{Ext}^{1}$-universal}
objects in the category of torsion abelian groups (see Theorem \ref{thm:main}).

The paper is organized as follows. In Section \ref{preliminares}
we introduce some notation and some facts about torsion pairs
in the category of abelian groups, Grothendiek categories, and extensions in abelian categories. In particular, we will review the properties
of the natural map $\Psi$ mentioned above, and recall the basic properties
of universal extensions. Among the results of this section,
it is worth to highlight Proposition \ref{prop:pres implica 2 pres},
in which we prove that, in a Grothendieck category $\mathcal{G}$,
for every $X\in\mathcal{G}$ there is a regular cardinal $\lambda$ such that
$\Ext[X][1][\G][-]:\G\to\Ab$ preserves \emph{$\lambda$-directed
colimits}. We will resume this result later when studying projective effacements. 

In Section \ref{Ab4} we study universal extensions in Ab3 abelian
categories. We will see that the existence of a universal extension
of $B$ by $A$ is closely related to the bijectivity of the canonical morphism
\[
\Psi:\operatorname{Ext}_{\mathcal{A}}^{1}(B^{(I)},A)\rightarrow\Ext[B][1][][A]^{I},
\]
for every set $I$ (see Lemma \ref{lem:phi es biyectiva}). This will
allow us to identify a special universal extension $\overline{\eta}$
with the property that the extension group $\operatorname{Ext}_{\mathcal{A}}^{1}(B^{(X)},A)$
turns out to be a cyclic right $\End[\mathcal{A}][B^{(X)}]$-module
generated by $\overline{\eta}$, where $X:=\operatorname{Ext}_{\mathcal{A}}^{1}(B,A)$
(see Theorem \ref{thm:canonica genera ext}). Finally in Theorem \ref{thm:Ab4 y Ext-peque sii ext universales}
we will see that an Ext-small Ab3 abelian category is Ab4 if and only
if every pair of objects in such category admits a universal extension.

In Section \ref{sec:eff} we study the projective effacements
in a Grothendieck category. Specifically, in Proposition \ref{prop:lambda-p2 vs proj eff}
we will use universal co-extensions to construct a projective effacement.
As a corollary, we obtain a new proof of Roos' theorem (see Corollary
\ref{cor:roos}). 

Section \ref{Ab3} is devoted to the study of the $\operatorname{Ext}^{1}$-universal
objects in an Ab3 abelian category. In particular,
we prove that the class of $\operatorname{Ext}^{1}$-universal
objects is closed under coproducts and direct summands (see Corollary
\ref{cor:coproducto de universales} and Corollary \ref{cor:sumandodirectodeext-universal}).

Finally, in Section \ref{torsion} we will seek for a characterization
of the co-$\operatorname{Ext}^{1}$-universal objects in the category $\mathcal{T}_{Z}$
of abelian torsion groups. We achieve the desired characterization in two steps. In the step one, in
Proposition \ref{lem:caracterizacion_existencia_coextuniv} we will
characterize the pairs of objects in $\mathcal{T}_{Z}$ that admit a universal co-extension. This will lead us to show that $\bigoplus_{n\geq1}\mathbb{Z}(p^{n})$
is not co-$\operatorname{Ext}^{1}$-universal in $\mathcal{T}_{Z}$ (see Corollary \ref{exa:contraejemplo }). In step two, exploiting the previous step and some facts from \emph{basic} \emph{subgroups theory}, we will show in Theorem \ref{thm:main} that a torsion group $V$ is a
co-$\operatorname{Ext}^{1}$-universal object in $\mathcal{T}_{Z}$
if and only if its reduced $p$-components are bounded.

\section{Preliminaries\label{preliminares}}

\subsection{Abelian groups}

We will denote by $\Ab$  the category of abelian groups. Recall that
$\Ab$ is a hereditary abelian category, that is, $\Ext[M][k][\Ab][N]=0$
for all $M,N\in\Ab$ and $k>1$. In particular, this means that every
quotient of an injective abelian group is also injective. We also
recall that every $G\in\Ab$ is expressed as a direct sum $G=D\oplus R$,
where $D$ is injective and $R$ is \textbf{reduced} (which means
that $R$ has no nonzero injective subobjects). A group $G\in\Ab$
is \textbf{bounded} if there is an integer $n\geq1$ such that $0=nG:=\{ng\,|\:g\in G\}$.
Finally, we point out that, for every positive integer $n$, $\mathbb{Z}(n)$
will denote the cyclic group of order $n$. For additional notation
and results the reader is referred to \cite{fuchs1960abelian,fuchs2015abelian}.

\subsection{Regular cardinals, directed sets and directed colimits}

Recall that a cardinal $\lambda$ is \textbf{regular} if it is an
infinite cardinal which is not a sum of less than $\lambda$ cardinals,
all smaller than $\lambda$. In such case, a poset is called \textbf{$\lambda$-directed}
if every subset of cardinality smaller than $\lambda$ has an upper
bound. Moreover, a $\aleph_{0}$-directed poset is called \textbf{directed
set}. On the other hand, recall that a category is \textbf{small}
when the isomorphism classes of its objects form a set. In this sense,
if $D$ is a $\lambda$-directed poset, for some regular cardinal
$\lambda$, then $D$ can be viewed as a small category whose objects
are the elements of $D$ and there is a unique morphism $\alpha\to\beta$
exactly when $\alpha\leq_{D}\beta$.

In what follows $D$ and $\C$ will denote a small category and a
category, respectively. A functor $D\to\C$ will be called a \textbf{$D$-diagram}
on $\C$. In this setting, the $D$-diagrams on $\C$ together with
the respective natural transformations form a category, which will
be denoted by $\Fun(D,\C)$. Moreover, the assignment $C\mapsto\kappa^{D}(C)$
underlies a functor $\kappa^{D}:\C\to\Fun(D,\C)$, where $\kappa^{D}(C)$
is the $D$-diagram on $\C$ such that $d\mapsto C$ and $f\mapsto1_{C}$,
for each object $d$ and morphism $f$ of $D$, and for each morphism
$h$ in $\C$, $\kappa_{h}^{D}:=\kappa^{D}(h)$ denote the natural
transformation given by the family of morphisms $(\kappa_{h,d}^{D})_{d\in D}$
with $\kappa_{h,d}^{D}=h$, for all $d\in D$. Such functor is called
the \textbf{constant diagram functor}.

Let $F$ be a functor of $\Fun(D,\C)$, $C$ be an object of $\C$,
and $\rho:F\to\kappa^{D}(C)$ be a natural transformation. Then, we
will say that the pair $(C,\rho)$ is a \textbf{$D$-colimit} of $F$
when the following condition holds: for each natural transformation
$\tau:F\to\kappa^{D}(C')$, there is a unique morphism $f:C\to C'$
in $\C$ such that $\kappa_{f}^{D}\circ\rho=\tau$. In such case,
we use the following notation $C=\mathsf{colim}_{D}(F)$ or $C=\mathsf{colim}_{D}F(d)$.
Now, if each $D$-diagram on $\C$ has a $D$-colimit, we say that
$\C$ is \textbf{$D$-co-complete}. In this case, the functor $\col_{D}:\Fun(D,\C)\to\C$
associating each $D$-diagram to its colimit is the left adjoint to
the constant diagram functor $\kappa^{D}$. Moreover, the category
$\C$ is called \textbf{co-complete} when it admits $D$-colimits
for every small category $D$. 
Dually, we say that the category $\C$ is \textbf{$D$-complete} (resp.
\textbf{complete}) if $\C^{op}$ is $D$-co-complete (resp. co-complete),
for $D$ a small category. We say that a category is \textbf{bicomplete}
if it is complete and co-complete. Lastly, let $\lambda$ be a regular
cardinal and suppose that $D$ is a $\lambda$-directed poset. A $D$-diagram
on $\C$ is called a \textbf{$\lambda$-directed diagram}, and its
colimit (whenever it exists) is called a \textbf{$\lambda$-directed
colimit}. 

\subsection{Grothendieck categories}

In the sequel $\A$ is an abelian category. Given $A,B\in\mathcal{A}$,
the symbol $\Ext[A][1][][B]$ denotes the class of all equivalence
classes of the short exact sequences in $\A$ of the form $\epsilon:\suc[B][E][A][\,][\,]$,
where $\hookrightarrow$ denotes a monomorphism in $\mathcal{A}$
and $\twoheadrightarrow$ an epimorphism in $\mathcal{A}$. In this
sense, we will denote by $\overline{\epsilon}$ the respective equivalence
class associated to $\epsilon$. We recall that, in general, the class
$\Ext[A][1][][B]$ is not necessarily a set (see \cite[Chapter 6, Exercise A]{freyd}).
When $\Ext[A][1][][B]$ is a set for all $A,B\in\A$, the abelian
category $\A$ is called \textbf{Ext-small}. Now, the morphisms in
$\mathcal{A}$ induce additive assignments between the extensions
in $\A$ via pullback and pushout. Namely, for each $f\in\Hom[\mathcal{A}][X][A]$
the assignment $\overline{\eta}\mapsto\overline{\eta\cdot f}$ (resp.
$\overline{\eta}\mapsto\overline{f\cdot\eta}$) induce an additive
`map' from $\Ext[A][1][][B]$ (resp. $\Ext[B][1][][X]$) to $\Ext[X][1][][B]$
(resp. $\Ext[B][1][][A]$).

On the other hand, recall that the category $\Fun(D,\A)$ is an abelian
category, for each small category $D$. Let us recall the `hierarchy'
among abelian categories introduced by Grothendieck (see \cite{Ab}).
Concretely, we say that $\A$ is: 
\begin{itemize}
\item \textbf{Ab3 }if all set-indexed coproducts exist in $\A$ (equivalently,
if it is co-complete); 
\item \textbf{Ab4 }if it is Ab3 and the functors $\col_{D}$ are exact,
for each set $D$ viewed as the (small) category whose objects are
the elements of $D$ and the only morphisms in it are the identity
morphisms; 
\item \textbf{Ab5 }if it is Ab3 and the functors $\col_{D}$ are exact,
for each $D$ directed set viewed as a small category. 
\item \textbf{Grothendieck} if it is Ab5 and it has a \textbf{generator},
i.e. an object $G$ in $\A$ such that the functor $\Hom[\A][G][-]$
is faithful.
\end{itemize}
We will denote by \textbf{Abn{*}} the property dual to Abn for each
$n\in\{3,4,5\}$. Recall that a Grothendieck category is automatically
bicomplete and it has enough injectives (see \cite[Corollary 2.8.9, Corollary 3.7.10, Theorem 3.10.10]{Popescu}). 
\begin{defn}
Let $\G$ be a Grothendieck category and let $\lambda$ be a regular
cardinal. Then, an object $X$ in $\G$ is said to be \textbf{$\lambda$-presentable}
if $\Hom[\G][X][-]:\G\to\mathsf{Ab}$ preserves $\lambda$-directed
colimits. In such case, we say that $X$ is \textbf{$\lambda$-FP2}
if $\Ext[X][1][\G][-]:\G\to\Ab$ preserves $\lambda$-directed colimits. 
\end{defn}

\begin{rem}
If $\lambda=\omega$, then $\lambda$-presentable objects are called
finitely presented and $\lambda$-FP2 objects are called FP2 (see
\cite{bravo2019locally}). For additional information on $\lambda$-presentable
objects, the reader is referred to \cite{adamek1994locally}.
\end{rem}

Note that every $\lambda$-presentable (resp. $\lambda$-FP2) object
in $\G$ is also $\mu$-presentable (resp. $\mu$-FP2), for each regular
cardinal $\mu>\lambda$. Now, we recall the following fact.
\begin{rem}
\label{rem:ordinal-cardinal-G} It is well-known that, for every Grothendieck
category $\G$, there exist a regular cardinal $\lambda_{\G}$ and
a set $\mathcal{S}$ of $\lambda_{\G}$-presentable objects in $\G$
such that every object in $\G$ is a $\lambda_{\G}$-directed colimit
of objects in $\mathcal{S}$ (i.e. the image of the respective $\lambda_{\G}$-directed
diagram lives in $\mathcal{S}$) (see \cite[Lemma 2.5.16]{krause2021homological}).
In this case, we have that every object in $\G$ is a $\mu$-directed
colimit of objects which are $\mu$-presentable, for each regular
cardinal $\mu>\lambda_{G}$ (see \cite[Lemma 2.5.13]{krause2021homological}). 
\end{rem}

The following two results are inspired in \cite[Lemma 2.2]{enochs2004flat}
and \cite[Proposition  2.3]{enochs2004flat} with the only difference
that we replace colimits of functors $F:\mu\rightarrow\mathcal{G}$,
that have an ordinal $\mu$ as their domain, with $\lambda$-directed
colimits where $\lambda$ is a regular cardinal. While the proof is
essentially the same, we consider worthwhile to state them as follows
because, if $\lambda>\omega$, the notion of $\lambda$-presentable
object cannot be defined using only colimits of functors $F:\mu\rightarrow\mathcal{A}$
that have an ordinal $\mu$ as their domain (see \cite[p.22]{adamek1994locally}). 
\begin{lem}
\label{lem: colimite lambda directo de inyectivos}Let $\mathcal{G}$
be a Grothendieck category. If we denote by $\Inj[\G]$ the class
of all injective objects in $\G$, then there exists a regular cardinal
$\lambda$ such that every $\lambda$-directed colimit of objects
in $\Inj[\G]$ is also an injective object of $\G$. 
\end{lem}

\begin{proof}
It is a known fact that every object in $\G$ is $\lambda$-presentable,
for some regular cardinal $\lambda$ (e.g. see Lemma 2.5.13 and Proposition
2.5.16 in \cite{krause2021homological}). Let $G$ be a generator
of $\mathcal{G}$. Since the lattice $\mathcal{L}(G):=\{\text{subobjects of }G\}$
is a set, there exist a regular cardinal $\lambda$ such that every
$S\in\mathcal{L}(G)$ is $\lambda$-presentable. Now, let $F:D\to\G$
be a $\lambda$-directed diagram on $\G$ such that each $F(d)$ is
an injective object of $\G$, for all $d\in D$. For each $S\in\mathcal{L}(G)$,
we consider the following commutative diagram in $\G$, where $\iota_{S}:S\to G$
denote the respective inclusion:\\
\noindent\begin{minipage}[t]{1\columnwidth}%
\[
\begin{tikzpicture}[-,>=to,shorten >=1pt,auto,node distance=1.5cm,main node/.style=,x=6cm,y=1.5cm]

   \node (1)  at (0,0)   {$\operatorname{Hom}_{\mathcal{G}}(G,\mathsf{colim}_D F(d))$};
   \node (2)  at (1,0)          {$\operatorname{Hom}_{\mathcal{G}}(S,\mathsf{colim}_D F(d))$};
   \node (3)  at (0,1)          {$\mathsf{colim}_D \operatorname{Hom}_{\mathcal{G}}(G,F(d))$};
   \node (4)  at (1,1)          {$\mathsf{colim}_D \operatorname{Hom}_{\mathcal{G}}(S,F(d))$};

\draw[->, thin]  (1)  to [below] node  {$\scriptstyle \operatorname{Hom}_{\G}(\iota_{S},\mathsf{colim}_{D}(F(d)))$} (2);
\draw[<-, thin]  (2)  to  node  {$$} (4);
\draw[<-, thin]  (1)  to  node  {$$} (3);
\draw[->, thin]  (3)  to  node  {$\scriptstyle \mathsf{colim}_{D}(\operatorname{Hom}_{\G}(\iota_{S},F(d)))$} (4);

\end{tikzpicture}
\]%
\end{minipage}\\
\medskip{}

Notice that $\mathsf{colim}_{D}(\operatorname{Hom}_{\G}(\iota_{S},F(d)))$
is an epimorphism since each $F(d)$ is an injective object of $\G$,
for all $d\in D$. And hence,  $\operatorname{Hom}_{\G}(\iota_{S},\mathsf{colim}_{D}(F(d)))$
is also an epimorphism (recall that $G$ and $S$ are $\lambda$-presentables).
Therefore, it follows from \cite[Proposition V.2.9]{ringsofQuotients}
that $\mathsf{colim}_{D}F(d)$ is an injective object of $\G$. 
\end{proof}
Finally, we highlight the following fact. 

\begin{prop}
\label{prop:pres implica 2 pres}Let $\mathcal{G}$ be a Grothendieck
category, $\lambda$ be a regular cardinal, and $X\in\mathcal{G}$.
If $X$ is $\lambda$-presentable, then there is a regular cardinal
$\mu\geq\lambda$ such that $X$ is $\mu$-FP2. 
\end{prop}

\begin{proof}
Let $X$ be a $\lambda$-presentable object of $\G$ and let $E$
be an injective cogenerator of $\G$. Consider the functor $\gamma:\G\to\G$,
mapping an object $X$ to the product $E^{\operatorname{Hom}_{\G}(X,E)}$.
In this case, such functor comes with a natural transformation $\nu:1_{\G}\to\gamma$,
which is monomorphic. On the other hand, let $\kappa$ be a regular
cardinal as in Lemma \ref{lem: colimite lambda directo de inyectivos}
and set $\mu:=\max\{\lambda,\kappa\}$. We claim that $X$ is $\mu$-FP2.
For this, it is enough to check that, the functor $\Ext[X][1][\G][-]$
preserves $\mu$-directed colimits. Indeed, let $F:D\to\G$ be a $\mu$-directed
diagram on $\G$ and consider the following short exact sequence in
$\Fun(D,\G)$ 
\[
\suc[F][\gamma\circ F][C][\nu_{F}]\text{,}
\]
where $\nu_{F}$ is the natural transformation induced by $\nu$.
Using now the Ab5 condition of $\G$, we get the following short exact
sequence in $\G$: 
\[
\suc[\col_{D}F][\col_{D}\left(\gamma\circ F\right)][\col_{D}C]\text{,}
\]
Applying the functor $(X,-):=\operatorname{Hom}_{\G}(X,-)$ and considering
the fact that $\mathsf{colim}_{D}(\gamma\circ F)\in\Inj[\G]$ (by
Lemma \ref{lem: colimite lambda directo de inyectivos}), we obtain
the following commutative diagram with exact rows, where $^{1}(X,-):=\operatorname{Ext}_{\G}^{1}(X,-)$:
\\
\noindent\begin{minipage}[t]{1\columnwidth}%
\[
\begin{tikzpicture}[-,>=to,shorten >=1pt,auto,node distance=1.5cm,main node/.style=,x=3.2cm,y=1cm]

   \node (1)  at (0,0)        {$(X,\mathsf{colim}_D F)$};
   \node (2)  at (1,0)        {$(X,\mathsf{colim}_D (\gamma \circ F))$};
   \node (3)  at (2,0)        {$(X,\mathsf{colim}_D C)$};
   \node (4)  at (3,0)        {$^{1}(X,\mathsf{colim}_D F)$};

   \node (1')  at (0,1)      {$\mathsf{colim}_D (X,F)$};
   \node (2')  at (1,1)      {$\mathsf{colim}_D (X,\gamma \circ F)$};
   \node (3')  at (2,1)      {$\mathsf{colim}_D (X,C)$};
   \node (4')  at (3,1)      {$\mathsf{colim}_D ( ^{1}(X,F))$};

\draw[right hook->, thin]  (1)  to  node  {$$} (2);
\draw[->, thin]  (2)  to  node  {$$} (3);
\draw[->>, thin]  (3)  to  node  {$$} (4);

\draw[->, thin]  (1')  to  node  {$$} (1);
\draw[->, thin]  (2')  to  node  {$$} (2);
\draw[->, thin]  (3')  to  node  {$$} (3);
\draw[->, thin]  (4')  to  node  {$$} (4);

\draw[right hook->, thin]  (1')  to  node  {$$} (2');
\draw[->, thin]  (2')  to  node  {$$} (3');
\draw[->>, thin]  (3')  to  node  {$$} (4');

\end{tikzpicture}
\]%
\end{minipage}\\
where the three vertical morphisms on the left are isomorphisms (since
$X$ is $\mu$-presentable) and therefore all vertical morphisms are
isomorphisms by the Five Lemma, as desired.
\end{proof}

\subsection{Torsion theories in $\protect\Ab$}

We are interested in highlighting some properties of the hereditary
torsion classes in $\Ab$. Such properties are crucial for our purposes. 
\begin{defn}
A \textbf{torsion pair} in $\Ab$ is a pair $\mathbf{t}:=(\T,\F)$
of full subcategories such that $\T=\{X\in\Ab:\operatorname{Hom}_{\Ab}(X,F)=0,\text{ for all }F\in\F\}$
and $\F=\{X\in\Ab:\operatorname{Hom}_{\Ab}(T,X)=0,\text{ for all }T\in\T\}$.
In such case, for each $X\in\Ab$ there is a (functorial) exact sequence
in $\Ab$ of the form: 
\[
\suc[T_{X}][X][F_{X}][\iota_{X}][\pi_{X}]
\]
where $T_{X}\in\T$ and $F_{X}\in\F$. The \textbf{torsion radical}
(resp. \textbf{coradical}) of $\mathbf{t}$ is the endofunctor on
$\Ab$ given by the assignment $X\mapsto T_{X}$ (resp. $X\mapsto F_{X}$).
Finally, $\mathcal{T}$ (resp. $\mathcal{F}$) is called the \textbf{torsion}
(resp. \textbf{torsionfree}) \textbf{class} with respect to $\mathbf{t}$. 
\end{defn}

A torsion pair $\mathbf{t}=(\mathcal{T,F})$ in $\Ab$ is called: 
\begin{enumerate}
\item \textbf{hereditary} if $\mathcal{T}$ is closed under subgroups (or,
equivalently, if $\mathcal{F}$ is closed under injective envelopes); 
\item \textbf{of finite type} if $\mathcal{F}$ is closed under directed
colimits, i.e. every $\aleph_{0}$-directed colimit of objects in
$\F$ is in $\mathcal{F}$. 
\end{enumerate}
Let us recall the following result that characterizes the hereditary
torsion pairs in $\Ab$ (see \cite[Corollary 3.11, Proposition 2.1]{bravo2019t},
\cite[Lemma 4.2]{silting2017}, and \cite[Chapter VI, Sections 5 and 6]{ringsofQuotients}).
Recall that $\mathsf{Spec}(\mathbb{Z})$ is the set of prime ideals
of $\mathbb{Z}$, $\mathsf{Supp}(M)$ is the set of prime ideals $P$
such that the localization of $M$ at $P$ is not zero, and $Z\subseteq\mathsf{Spec}(\mathbb{Z})$
is a \textbf{sp-subset} if for any prime ideals $P\subseteq Q$, with
$P\in Z$, we have that $Q\in Z$. 
\begin{prop}
Let $\mathbf{t}=(\mathcal{T},\mathcal{F})$ be a torsion pair in $\Ab$.
Then, the following assertions are equivalent: 
\begin{enumerate}
\item $\mathbf{t}$ is a hereditary torsion pair; 
\item $\mathbf{t}$ is of finite type;
\item $\mathcal{T}=\mathcal{T}_{Z}:=\{M\in\Ab:\mathsf{Supp}(M)\subseteq Z\}$,
for some sp-subset $Z$ of $\mathsf{Spec}(\mathbb{Z})$. 
\end{enumerate}
\end{prop}

The following examples will be crucial in the paper. 
\begin{example}
\label{exa:tor-her-classic} Let $p$ be a prime number and, we put
$Z:=\mathsf{Spec}(\mathbb{Z})\setminus\{0\}$ and $Z_{p}:=\{p\mathbb{Z}\}$.
Note that $Z$ and $Z_{p}$ are sp-subsets of $\mathsf{Spec}(\mathbb{Z})$.
In these cases, we have the following: 
\begin{enumerate}
\item the class $\mathcal{T}_{Z}=\{M\in\Ab:M\text{ is a torsion group}\}$
is the torsion class of a hereditary torsion pair in $\Ab$; 
\item the class $\mathcal{T}_{p}:=\mathcal{T}_{Z_{p}}=\{M\in\Ab:\forall m\in M,\exists n_{m}\geq0\text{ such that }o(m)=p^{n_{m}}\}$
is the torsion class of a hereditary torsion pair in $\Ab$. 
\end{enumerate}
\end{example}

We finish this subsection with the following remark and example. 
\begin{rem}
\label{rem:Tor-her-Gro} Every torsion class of a hereditary torsion
pair in $\Ab$ is clearly an abelian exact subcategory of $\Ab$.
Moreover, it is a Grothendieck category, where the coproducts in it
are computed as in $\Ab$, but the products in it are not computed
as in $\Ab$ (see Remark \ref{rem:propertiesoft} and \cite[Corollary 4.3(1,3)]{parra2021tilting}). 
\end{rem}

\begin{example}
\label{exam:no-Ab4*} The Grothendieck categories $\mathcal{T}_{Z}$
and $\mathcal{T}_{Z_{p}}$ are not an Ab4{*} abelian category. Indeed,
note that the product in $\mathcal{T}_{Z}$ (resp. $\mathcal{T}_{p}$)
of the family of the canonical epimorphisms $\{f_{n}:\mathbb{Z}_{p^{n}}\to\mathbb{Z}_{p}\}_{n\geq1}$
can not be surjective. 
\end{example}


\subsection{Extensions, coproducts, and products\label{subsec:ext,prod,coprod}}

Let $\mathcal{A}$ be an abelian category and let $\{A_{i}\}_{i\in I}$
be a set of objects in $\mathcal{A}$ whose product (resp. coproduct)
exists in $\A$. Now, we denote by $\pi_{i}^{A}:\prod_{i\in I}A_{i}\to A_{i}$
(resp. $\mu_{i}^{A}:A_{i}\to\coprod_{i\in I}A_{i}$) the associated
$i$-th projection (resp. inclusion). When there is an object $A$
in $\A$ such that $A_{i}=A$, for all $i\in I$, there exists a unique
morphism $\Delta_{I}^{A}:A\rightarrow A^{I}$ (resp. $\nabla_{I}^{A}:A^{(I)}\rightarrow A$)
such that $\pi_{i}^{A}\circ\Delta_{I}^{A}=1_{A}$ (resp. $\text{\ensuremath{\nabla_{I}^{A}}}\circ\mu_{i}^{A}=1_{A}$).
We will refer to such morphism as the $I$-\textbf{diagonal} (resp.
$I$-\textbf{co-diagonal}) morphism associated to $A$. On the other
hand, for each object $B$ in $\A$ there is a natural map: 
\begin{gather*}
\Psi:\operatorname{Ext}_{\mathcal{A}}^{1}(\bigoplus_{i\in I}A_{i},B)\rightarrow\prod_{i\in I}\Ext[A_{i}][1][][B]\\
\left(\mbox{resp. }\Phi:\operatorname{Ext}_{\mathcal{A}}^{1}(B,\prod_{i\in I}A_{i})\rightarrow\prod_{i\in I}\Ext[B][1][][A_{i}]\right)
\end{gather*}
defined as $\Psi(\overline{\epsilon}):=\left(\overline{\epsilon\cdot\mu_{i}^{A}}\right)_{i\in I}$
(resp. $\Phi(\overline{\epsilon}):=(\overline{\pi_{i}^{A}\cdot\epsilon})_{i\in I}$).
Now, we recall the main result in \cite{argudin2021yoneda}. 

It is important to mention three aspects of the map $\Psi$. The first
one: $\Psi$ is always injective, for all set of objects in $\A$
(see \cite[Lemma 4.2]{argudin2022exactness}); the second one: when
$\Psi$ is bijective, we can exhibit the correspondence rule of $\Psi^{-1}$,
as shown in the following lemma; and the third one: when $\mathcal{A}$
is Ab3, $\mathcal{A}$ is Ab4 if and only if $\Psi$ is always bijective,
as shown in the theorem below. 
\begin{lem}
\label{lem:la extension colimite mediante la codiagonal }\label{lem:colim}\label{lem:la inversa}
Let $\mathcal{A}$ be an Ab3 (resp. Ab3{*}) abelian category, $A$
be an object in $\mathcal{A}$, $\{B_{i}\}_{i\in I}$ be a set of
objects in $\mathcal{A}$, $\eta=\left\{ \eta_{i}:\:\suc[A][E_{i}][B_{i}][f_{i}][g_{i}]\right\} _{i\in I}$
be a set of short exact sequences in $\mathcal{A}$, and $\nabla_{I}^{A}$
be the $I$-co-diagonal morphism. Consider the following pushout diagram.
\\
\noindent\begin{minipage}[t]{1\columnwidth}%
\[ 
\begin{tikzpicture}[-,>=to,shorten >=1pt,auto,node distance=2cm,main node/.style=,x=1.2cm,y=1.2cm]

   \node[main node] (1) at (0,0)      {$$};     
  \node[main node] (2) [right of=1]  {$A^{(I)}$};     
  \node[main node] (3) [right of=2]  {$\bigoplus E_i $};    
   \node[main node] (4) [right of=3]  {$\bigoplus B_i $}; 
  \node[main node] (5) [right of=4]  {$$};

   \node[main node] (1') at (0,-1)       {$$};    
   \node[main node] (2') [right of=1']  {$A$};    
   \node[main node] (3') [right of=2']  {$Z_\eta$};   
    \node[main node] (4') [right of=3']  {$\bigoplus B_i $};   
    \node[main node] (5') [right of=4']  {$$};


\draw[->, thin]   (2)  to node  {$\scriptstyle \bigoplus  f_i$}  (3);
\draw[->>, thin]   (3)  to node  {$\scriptstyle \bigoplus g_i$}  (4);


\draw[->, thin]   (2')  to node  {$\scriptstyle f_\eta$}  (3');
\draw[->>, thin]   (3')  to node  {$\scriptstyle g_\eta$}  (4');
\draw[->, thin]   (2)  to node  {$\nabla ^{A} _I$}  (2');
\draw[->, thin]   (3)  to node  {$\nabla '$}  (3');
\draw[-, double]   (4)  to node  {$$}  (4');

    \end{tikzpicture}
\]%
\end{minipage}\\
Then, the following statements hold true: 
\begin{enumerate}
\item There is a short exact sequence $\eta':\:\suc[A][X][\bigoplus_{i\in I}B_{i}][f'][g']$
such that $\Psi(\overline{\eta'})=(\overline{\eta_{i}})_{i\in I}$
if and only if $f_{\eta}$ is a monomorphism. 
\item If $f_{\eta}$ is a monomorphism, then any short exact sequence $\eta'$
satisfying that $\Psi(\overline{\eta'})=(\overline{\eta_{i}})_{i\in I}$
is equivalent to $\eta'':\:\suc[A][Z_{\eta}][\bigoplus_{i\in I}B_{i}][f_{\eta}][g_{\eta}]$. 
\item If $\mathcal{A}$ is Ab4, then $\Psi^{-1}((\overline{\eta}_{i})_{i\in I})=\overline{\nabla_{I}^{A}\cdot\left(\bigoplus_{i\in I}\eta_{i}\right)}$. 
\end{enumerate}
\end{lem}

\begin{proof}
$ $
\begin{enumerate}
\item $(\Rightarrow)$ Suppose that $\Psi(\overline{\eta'})=(\overline{\eta_{i}})_{i\in I}$.
Thus, we have the following commutative diagram, for all $i\in I$:\\
\noindent\begin{minipage}[t]{1\columnwidth}%
\[ \begin{tikzpicture}[-,>=to,shorten >=1pt,auto,node distance=2cm,main node/.style=,x=1.2cm,y=1.2cm]
   \node[main node] (0) at (-1,0)      {$\eta _i:$}; 
      \node[main node] (1) at (0,0)      {$A$};    
   \node[main node] (2) [right of=1]  {$E_i$};  
     \node[main node] (3) [right of=2]  {$B_i$};

   \node[main node] (0') at (-1,-1)   {$\eta':$};   
    \node[main node] (1') at (0,-1)   {$A$};     
  \node[main node] (2') [right of=1']  {$X$};    
   \node[main node] (3') [right of=2']  {$\bigoplus _{i\in I} B_i $};

\draw[right hook->, thin]   (1)  to node  {$f_i$}  (2);
\draw[->>, thin]   (2)  to node  {$g_i$}  (3);
\draw[right hook->, thin]   (1')  to node  {$f'$}  (2');
\draw[->>, thin]   (2')  to node  {$g'$}  (3');
\draw[-, double]   (1)  to node  {$$}  (1');
\draw[->, thin]   (2)  to node  {$\mu  _i'$}  (2');
\draw[->, thin]   (3)  to node  {$\mu ^B _i$}  (3');
    \end{tikzpicture} \]
\end{minipage}\\
Now, by the universal property of coproducts there is a morphism $\gamma\in\Hom[][\bigoplus_{i\in I}E_{i}][X]$
such that $\mu'_{i}=\gamma\circ\mu_{i}^{E}$, for all $i\in I$. In
particular, note that 
\[
\gamma\circ\left(\bigoplus_{i\in I}f_{i}\right)\circ\mu_{i}^{A}=\gamma\circ\mu_{i}^{E}\circ f_{i}=\mu'_{i}\circ f_{i}=f'=f'\circ\nabla_{I}^{A}\circ\mu_{i}^{A},\;\ \forall i\in I\mbox{.}
\]
Once again from the universal property of coproducts, we get that
$\gamma\circ\left(\bigoplus_{i\in I}f_{i}\right)=f'\circ\nabla_{I}^{A}$.
Then, by the universal property of pushouts, there is a morphism $\gamma':Z_{\eta}\rightarrow X$
such that $\gamma'\circ f_{\eta}=f'$ and $\gamma'\circ\nabla'=\gamma$.
And hence, $f_{\eta}$ is a monomorphism since $f'$ is a monomorphism.
\\
 $(\Leftarrow)$ Observe that we have the following commutative diagram
with exact rows, for all $i\in I$:\\
\noindent\begin{minipage}[t]{1\columnwidth}%
\[ 
\begin{tikzpicture}[-,>=to,shorten >=1pt,auto,node distance=2.5cm,main node/.style=,x=1.2cm,y=1.2cm]
       \node[main node] (2) at (0,0)  {$A^{(I)}$};
       \node[main node] (3) [right of=2]  {$\bigoplus E_i $}; 
      \node[main node] (4) [right of=3]  {$\bigoplus B_i $};    

     \node[main node] (2') at (0,-1)   {$ A$};     
  \node[main node] (3') [right of=2']  {$ Z_\eta$};   
    \node[main node] (4') [right of=3']  {$\bigoplus B_i $};   

 \node[main node] (2'') at (0,1)   {$A$};      

\node[main node] (3'') [right of=2'']  {$E_i$};    
   \node[main node] (4'') [right of=3'']  {$B_i $};    
   \node[main node] (A) at (-1,-1)  {$\eta '':$};    
   \node[main node] (B) at (-1,1)  {$\eta _i :$};
\draw[->, thin]   (2)  to node  {$\scriptstyle \bigoplus  f_i$}  (3);
\draw[->>, thin]   (3)  to node  {$\scriptstyle \bigoplus  g_i$}  (4);
\draw[right hook->, thin]   (2')  to node  {$\scriptstyle   f_\eta$}  (3');
\draw[->>, thin]   (3')  to node  {$\scriptstyle   g_\eta$}  (4');
\draw[right hook->, thin]   (2'')  to node  {$\scriptstyle   f_i$}  (3'');
\draw[->>, thin]   (3'')  to node  {$\scriptstyle  g_i$}  (4'');
\draw[->, thin]   (2)  to node  {$  \nabla^{A}_I$}  (2');
\draw[->, thin]   (3)  to node  {$  \nabla '$}  (3');
\draw[-, double]   (4)  to node  {$$}  (4');
\draw[->, thin]   (2'')  to node  {$\mu_i ^{ A}$}  (2);
\draw[->, thin]   (3'')  to node  {$\mu_i ^E$}  (3);
\draw[->, thin]   (4'')  to node  {$\mu_i ^B$}  (4);
    \end{tikzpicture} 
\]
\end{minipage} \\
 Then, since that the top and the bottom rows are short exact sequences
together with the fact that $\nabla_{I}^{A}\circ\mu_{i}^{A}=1_{A}$
$\forall i\in I$, it follows that $\Psi(\overline{\eta''})=(\overline{\eta_{i}})_{i\in I}$. 
\item It was proved above that $\Psi(\overline{\eta''})=(\overline{\eta_{i}})_{i\in I}$
when $f_{\eta}$ is monic. Therefore, (b) follows from \cite[Lemma 4.2]{argudin2022exactness}. 
\item When $\A$ is Ab4, we get that the sequence $\eta'':\:\suc[A][Z_{\eta}][\bigoplus B_{i}][f_{\eta}][g_{\eta}]$
is exact and, by the proof of the item (a), $\Psi(\overline{\eta''})=(\overline{\eta_{i}})_{i\in I}$.
Therefore, by Theorem \ref{teo:secondmain}, $\Psi^{-1}((\overline{\eta_{i}})_{i\in I})=\overline{\eta''}=\overline{\nabla_{I}^{A}\cdot\bigoplus_{i\in I}\eta_{i}}$. 
\end{enumerate}
\end{proof}
\begin{thm}
\cite[Theorem 4.8]{argudin2021yoneda}\label{thm:Ab4 vs ext}\label{teo:secondmain}
Let $\mathcal{A}$ be an Ab3 (resp. Ab3{*}) abelian category. Then,
$\mathcal{A}$ is Ab4 (resp. Ab4{*}) if and only if $\Psi$ (resp.
$\Phi$) is always bijective, for any set of objects in $\A$. 
\end{thm}

\subsection{Universal Extensions\label{subsec:Universal-Extensions}}
\begin{defn}
\cite[Definition 5.6]{parra2021tilting}\label{def:extension universal}
Let $\mathcal{A}$ be an abelian category and $A,B\in\mathcal{A}$.
A \textbf{universal extension} of $B$ by $A$ is a short  exact sequence
\[
\suc[A][E][B^{(X)}][u][p]
\]
in $\mathcal{A}$, for some non-empty set $X$, such that one of the
following equivalent statements hold true:
\begin{enumerate}
\item $\Ext[B][1][][u]:\Ext[B][1][][A]\rightarrow\Ext[B][1][][E]$ is the
zero morphism, 
\item $\Ext[B][1][][p]:\Ext[B][1][][E]\rightarrow\Ext[B][1][][B^{(X)}]$
is injective, 
\item the connection morphism $\delta:\Hom[][B][B^{(X)}]\rightarrow\Ext[B][1][][A]$
is surjective. 
\end{enumerate}
\end{defn}

For the keen reader, the term \emph{extension} in the above definition
may not seem precise. Perhaps a more appropriate definition would
be that a universal extension is the equivalence class of an exact
sequence that satisfies any of the above conditions. However, for
the sake of simplicity, we will keep the above definition. 
\begin{rem}
\label{rem:extensiones universales} Let $\mathcal{A}$ be an abelian
category and $A,B\in\mathcal{A}$.
\begin{enumerate}
\item If $B$ is projective, then every exact sequence $\suc[A][E][B^{(X)}][u][p]$
is a universal extension of $B$ by $A$. 
\item \cite[Proposition  5.7]{parra2021tilting} If $\Ext[A][1][][B]$ is
a finitely generated $\End[\mathcal{\mathcal{A}}][B]$-module, then
there is a universal extension of $B$ by $A$. 
\item \cite[Proposition  5.9]{parra2021tilting} If $\mathcal{A}$ is Ab4
and $\Ext[A][1][][B]$ is a set, then there is a universal extension
of $B$ by $A$. Indeed, let $\{\eta_{i}:\;\suc[A][E_{i}][B][f_{i}][g_{i}]\}_{i\in I}$
be a complete set of representatives of $\Ext[B][1][][A]$. It follows
from Theorem \ref{thm:Ab4 vs ext} and Lemma \ref{lem:la inversa}
that an exact sequence representing $\Psi^{-1}\left((\overline{\eta}_{i})_{i\in I}\right)$
is a universal extension of $B$ by $A$. 
\item \cite[Proposition  5.9]{parra2021tilting} If there is a universal
extension of $B$ by $A$, then $\Ext[B][1][][A]$ is a set. Indeed,
this follows from condition (c) of Definition \ref{def:extension universal}.
Examples of abelian categories where one can find objects $A$ and
$B$ such that $\Ext[B][1][][A]$ is not a set can be found in \cite[Chapter 6, Excercise A]{freyd}
and \cite[Lemma 1.1]{casacuberta2008brown}. 
\end{enumerate}
\end{rem}

\begin{problem}
Let $\mathcal{A}$ be an $\operatorname{Ext}$-small Ab3 abelian category.
Is the Ab4 condition equivalent to the existence of universal extensions? 
\end{problem}

In the next section we will show that the answer to this question
is affirmative. Then, the next question arises. 
\begin{problem}
Let $\mathcal{A}$ be an $\operatorname{Ext}$-small abelian category
satisfying Ab3 but not Ab4. Is it possible to characterize the objects
that admit universal extensions? 
\end{problem}

It is worth mentioning that throughout the article we will also be
interested in studying the dual notion of universal extension. We
will now present this notion for completeness. 
\begin{defn}
\label{def:extension co-universal} Let $\mathcal{A}$ be an abelian
category and $A,B\in\C$. A \textbf{universal co-extension} of $B$
by $A$ is a short exact sequence 
\[
\suc[B^{X}][E][A][p][u]
\]
in $\A$, for some non-empty set $X$, such that one of the following
equivalent statements hold true:
\begin{enumerate}
\item $\Ext[u][1][][B]:\Ext[A][1][][B]\rightarrow\Ext[E][1][][B]$ is the
zero morphism, 
\item $\Ext[p][1][][B]:\Ext[E][1][][B]\rightarrow\Ext[B^{X}][1][][B]$ is
injective, 
\item the connection morphism $\delta:\Hom[][B^{X}][B]\rightarrow\Ext[A][1][][B]$
is surjective. 
\end{enumerate}
\end{defn}

\section{Ab4 vs. Universal extensions\label{Ab4}}

This section contains the main results of the article. Specifically,
we will study the behavior of universal extensions in abelian Ab3
categories, and then characterize $\operatorname{Ext}$-small Ab4
abelian categories through the existence of universal extensions. 

\subsection{Universal extensions in Ab3 abelian categories}
\begin{lem}
\label{lem:phi es biyectiva}Let $\mathcal{A}$ be an Ab3 abelian
category and $A,B\in\mathcal{A}$. If there is a universal extension
of $B$ by $A$, then $\Psi:\Ext[B^{(X)}][1][][A]\rightarrow\Ext[B][1][][A]^{X}$
is bijective for every set $X$. In particular, $\Ext[B^{(X)}][1][][A]$
is a set for every set $X$. 
\end{lem}

\begin{proof}
Let $\eta:\:\suc[A][E][B^{(Y)}]$ be a universal extension of $B$
by $A$. By \cite[Lemma 4.2]{argudin2022exactness}, it is enough
to show that $\Psi$ is surjective. Let $(\overline{\eta_{i}})_{i\in X}\in\Ext[B][1][][A]^{X}$.
By Definition \ref{def:extension universal}(c), we know that for
every $i\in X$ there is a morphism $u_{i}:B\rightarrow B^{(Y)}$
such that $\overline{\eta\cdot u_{i}}=\overline{\eta_{i}}$. Now,
by the universal property of coproducts, gives a unique morphism $u:B^{(X)}\rightarrow B^{(Y)}$
such that $u\circ\mu_{i}^{B}=u_{i}$, for all $i\in X$. We claim
that $\overline{\eta'}:=\overline{\eta\cdot u}$ satisfies that $\Psi\left(\overline{\eta'}\right)=(\overline{\eta_{i}})_{i\in I}$.
Indeed, we have $\overline{\eta'\cdot\mu_{i}^{B}}=\overline{\eta\cdot u_{i}}=\overline{\eta_{i}}$.
Therefore, $\Psi$ is bijective as desired.

The final part follows from Remark \ref{rem:extensiones universales}(d). 
\end{proof}
The following result is a direct consequence of the previous lemma
together with the proof of the Lemma \ref{lem:la inversa}(a), for
the set mentioned below. 
\begin{cor}
\label{cor:caract ext univ} Let $\mathcal{A}$ be an Ab3 abelian
category and $A,B\in\mathcal{A}$ be objects. Then, there exists a
universal extension of $B$ by $A$ if, and only if, $X:=\Ext[B][1][][A]$
is a set and the natural map $\Psi:\Ext[B^{(X)}][1][][A]\rightarrow\Ext[B][1][][A]^{X}$
is bijective. 
\end{cor}

Before proceeding with the main theorem of this section, we prove
the following useful and interesting results. 
\begin{cor}
\label{cor:extension universal canonica}Let $\mathcal{A}$ be an
Ab3 abelian category, and $A,B\in\mathcal{A}$ such that $\Ext[A][1][][B]=\{\overline{\eta_{i}}\}_{i\in X}$
is a set. If there is a universal extension of $B$ by $A$, then
we can build $\eta:\:\suc[A][E][B^{(X)}]$ a universal extension of
$B$ by $A$ such that $\overline{\eta\cdot\mu_{i}^{B}}=\overline{\eta_{i}}$,
for all $i\in X$. Moreover, $\eta$ is the exact sequence described
in Lemma \ref{lem:colim}. 
\end{cor}

\begin{proof}
By Lemma \ref{lem:phi es biyectiva}, $\Psi:\Ext[B^{(X)}][1][][A]\rightarrow\Ext[B][1][][A]^{X}$
is bijective. Therefore, there is $\overline{\eta}\in\Ext[B^{(X)}][1][][A]$
such that $\Psi(\overline{\eta})=(\overline{\eta_{i}})_{i\in X}$.
This means that $\overline{\eta\cdot\mu_{i}^{B}}=\overline{\eta_{i}}$,
for all $i\in X$. 
\end{proof}
\begin{defn}
Let $\mathcal{A}$ be an Ab3 abelian category, $A,B\in\mathcal{A}$
be objects that admit a universal extensions of $B$ by $A$, and
$\eta$ the universal extension built in Corollary \ref{cor:extension universal canonica}.
We will say that $\eta$ is the \textbf{canonical universal extension}
of $B$ by $A$. 
\end{defn}

\begin{thm}
\label{thm:canonica genera ext}Let $\mathcal{A}$ be an Ab3 abelian
category, $A,B\in\mathcal{A}$ be objects that admit a universal extension
of $B$ by $A$, and $\eta:\:\suc[A][E][B^{(X)}][f][g]$ the canonical
universal extension of $B$ by $A$. Then, $\Ext[B^{(X)}][1][][A]$
is a cyclic right $\End[\mathcal{A}][B^{(X)}]$-module generated by
$\overline{\eta}$. 
\end{thm}

\begin{proof}
Let $\{\eta_{i}:\:\suc[A][E_{i}][B][f_{i}][g_{i}]\}_{i\in X}$ be
a complete set of representatives of $\Ext[B][1][][A]$ (see Remark
\ref{rem:extensiones universales}(d)). Given an extension $\overline{\eta'}\in\Ext[B^{(X)}][1][][A]$
with $\eta':\;\suc[A][E'][B^{(X)}][f'][g']$, we can define a function
$\sigma:X\rightarrow X$ such that $\overline{\eta'\cdot\mu_{i}^{B}}=\overline{\eta_{\sigma(i)}}$,
for all $i\in X$. From the aforementioned, we have the following
commutative diagram with exact rows, for every $i\in X$: \\
\noindent\begin{minipage}[t]{1\columnwidth}%
\[ \begin{tikzpicture}[-,>=to,shorten >=1pt,auto,node distance=1.5cm,main node/.style=,x=1.5cm,y=1.5cm]
   \node[main node] (E) at (.5,0)      {$\eta:$};    
   \node[main node] (E') [below of=E]      {$\eta _{\sigma (i)}:$};  
     \node[main node] (E'') [below of=E']   {$\eta ':$};

   \node[main node] (1) at (0,0)      {$$};     

 \node[main node] (2) [right of=1]  {$A$};   
    \node[main node] (3) [right of=2]  {$E$}; 
      \node[main node] (4) [right of=3]  {$B ^{(X)}$};    
   \node[main node] (5) [right of=4]  {$$};
   \node[main node] (1') [below of=1]      {$$};    
   \node[main node] (2') [right of=1']  {$A$};    
   \node[main node] (3') [right of=2']  {$E_{\sigma (i)}$};  
     \node[main node] (4') [right of=3']  {$B$};    
   \node[main node] (5') [right of=4']  {$$};

   \node[main node] (1'') [below of=1']      {$$};      
 \node[main node] (2'') [right of=1'']  {$A$};     
  \node[main node] (3'') [right of=2'']  {$E'$};      
 \node[main node] (4'') [right of=3'']  {$B^{(X)}$};      
 \node[main node] (5'') [right of=4'']  {$$};
\draw[right hook ->, thin]   (2)  to [above] node  {$f$}  (3);
\draw[->>, thin]   (3)  to [above] node  {$g$}  (4);
\draw[right hook ->, thin]   (2')  to [below] node  {$f_{\sigma (i)}$}  (3');
\draw[->>, thin]   (3')  to [below] node  {$g_{\sigma (i)}$}  (4');
\draw[right hook->, thin]   (2'')  to [below] node  {$f'$}  (3'');
\draw[->>, thin]   (3'')  to [below] node  {$g'$}  (4'');
\draw[-, double]   (2')  to node  {$$}  (2);
\draw[->, thin]   (3')  to node  {$\mu' _{\sigma (i)}$}  (3);
\draw[->, thin]   (4')  to node  {$\mu ^B _{\sigma (i)}$}  (4);
\draw[-, double]   (2')  to node  {$$}  (2'');
\draw[->, thin]   (3')  to node  {$u_i$}  (3'');
\draw[->, thin]  (4')  to node  {$\mu ^B _i$}  (4'');
    \end{tikzpicture} 
\]
\end{minipage}\\
 Now, by a similar argument to the proof of Lemma \ref{lem:colim}(c),
we have the following pushout diagram. \\
\noindent\begin{minipage}[t]{1\columnwidth}%
\[ \begin{tikzpicture}[-,>=to,shorten >=1pt,auto,node distance=2.5cm,main node/.style=,x=1.2cm,y=1.2cm]

   \node[main node] (0) at (-2,0)      {$\bigoplus _{i\in X} \eta _{\sigma (i)}:$};  
     \node[main node] (1) at (0,0)      {$A^{(X)}$};     
  \node[main node] (2) [right of=1]  {$\bigoplus _{i\in X} E _{\sigma (i)}$};    
   \node[main node] (3) [right of=2]  {$B^{(X)}$};   
    \node[main node] (0') at (-2,-1)   {$\eta':$};    
   \node[main node] (1') at (0,-1)   {$A$}; 
      \node[main node] (2') [right of=1']  {$E'$};    
   \node[main node] (3') [right of=2']  {$B^{(X)}$};

\draw[->, thin]   (1)  to node  {$\bigoplus _{i\in X} f _{\sigma (i)}$}  (2);
\draw[->>, thin]   (2)  to node  {$\bigoplus _{i\in X} g _{\sigma (i)}$}  (3);
\draw[right hook->, thin]   (1')  to node  {$f'$}  (2');
\draw[->>, thin]   (2')  to node  {$g'$}  (3');
\draw[->, thin]   (1)  to node  {$\nabla_X^{A}$}  (1');
\draw[->, thin]   (2)  to node  {$\nabla'$}  (2');
\draw[-, double]   (3)  to node  {$$}  (3');
    \end{tikzpicture} \]
\end{minipage} \\
Then, by the universal property of coproducts, there is $\gamma\in\Hom[][\bigoplus_{i\in X}E_{\sigma(i)}][E]$
such that $\gamma\circ\mu_{i}^{E_{\sigma}}=\mu'_{\sigma(i)}$, where
$\mu_{i}^{E_{\sigma}}:E_{\sigma(i)}\rightarrow\bigoplus_{i\in X}E_{\sigma(i)}$
is the $i$-th canonical inclusion. From the universal property of
coproducts and of the following chain of compositions, we deduce that
$f\circ\nabla_{X}^{A}=\gamma\circ(\bigoplus_{i\in X}f_{\sigma(i)})$
\[
f\circ\nabla_{X}^{A}\circ\mu_{i}^{A}=f=\mu'_{\sigma(i)}\circ f_{\sigma(i)}=\gamma\circ\mu_{i}^{E_{\sigma}}\circ f_{\sigma(i)}=\gamma\circ(\bigoplus_{i\in X}f_{\sigma(i)})\circ\mu_{i}^{A}\:\mbox{,}\forall i\in X\text{.}
\]
 Hence, it follows from the universal property of pushouts that there
is a morphism $\gamma':E'\rightarrow E$ such that $\gamma'\circ f'=f$.
And hence, by the universal property of cokernels, there is a morphism
$\gamma'':B^{(X)}\rightarrow B^{(X)}$ such that $\overline{\eta'}=\overline{\eta\cdot\gamma''}$. 
\end{proof}

\subsection{A characterization of Ab4 categories\label{sub.Ab4} }
\begin{thm}
\label{thm:Ab4 y Ext-peque sii ext universales} Let $\mathcal{A}$
be an $\operatorname{Ext}$-small abelian category which is Ab3. Then,
$\mathcal{A}$ is Ab4 if, and only if, there is a universal extension
of $B$ by $A$, for all $A,B\in\mathcal{A}$.
\end{thm}

\begin{proof}
By \cite[Proposition  5.9]{parra2021tilting}, we know that if $\mathcal{A}$
is Ab4 and $\operatorname{Ext}$-small, then universal extensions
exist for any pair of objects.

To prove the opposite implication, we consider an object $A\in\mathcal{A}$,
a set $\{B_{i}\}_{i\in X}$ of objects in $\mathcal{A}$ and its coproduct
$B:=\bigoplus_{i\in X}B_{i}$. We will seek to show that, for every
$(\overline{\eta_{i}})\in\prod_{i\in X}\Ext[B_{i}][1][][A]$, the
morphism $f_{H}$ in the exact sequence $A\stackrel{f_{H}}{\rightarrow}Z_{H}\stackrel{g_{H}}{\twoheadrightarrow}\bigoplus_{i\in X}B_{i}$
is a monomorphism, where $H=\{\eta_{i}\}_{i\in X}$ (see Lemma \ref{lem:la inversa}).
For this, we will use a universal extension of $B$ by $A$, say $\eta:\:\suc[A][D][B^{(Y)}][a][b]$.
Consider the exact sequence $\eta'_{i}:=\eta_{i}\oplus\kappa_{i}$,
for all $i\in X$, where $\kappa_{i}:\;\suc[0][\bigoplus_{j\in X-\{i\}}B_{i}][\bigoplus_{j\in X-\{i\}}B_{i}][][1]\mbox{.}$
By Definition \ref{def:extension universal}(c) there is a morphism
$u_{i}:B\rightarrow B^{(Y)}$ such that $\overline{\eta\cdot u_{i}}=\overline{\eta'_{i}}$
for every $i\in X$. This give us the following commutative diagram\\
\noindent\begin{minipage}[t]{1\columnwidth}%
\[ \begin{tikzpicture}[-,>=to,shorten >=1pt,auto,node distance=1.5cm,main node/.style=,x=1.5cm,y=1.5cm]
   
\node[main node] (0) at (0.5,0)      {$\eta _i:$};    
   \node[main node] (1) at (0,0)      {$$};     
  \node[main node] (2) [right of=1]  {$A$};   
    \node[main node] (3) [right of=2]  {$E_i$};   
    \node[main node] (4) [right of=3]  {$B_i$};   
    \node[main node] (5) [right of=4]  {$$}; 
          \node[main node] (0')  [below of=0]   {$\eta ' _i:$};   
    \node[main node] (1') [below of=1]      {$$};    
   \node[main node] (2') [right of=1']  {$A$};     
  \node[main node] (3') [right of=2']  {$E_i \oplus B' _i$};  
     \node[main node] (4') [right of=3']  {$B$};    
   \node[main node] (5') [right of=4']  {$$};    

   \node[main node] (0'')  [below of=0']   {$\eta :$};     
  \node[main node] (1'') [below of=1']      {$$};   
    \node[main node] (2'') [right of=1'']  {$A$};       
\node[main node] (3'') [right of=2'']  {$D$};       
\node[main node] (4'') [right of=3'']  {$B ^{(Y)}$};      
 \node[main node] (5'') [right of=4'']  {$$};

\draw[right hook ->, thin]   (2)  to node  {$f_i$}  (3);
\draw[->>, thin]   (3)  to node  {$g_i$}  (4);
\draw[right hook->, thin]   (2')  to node  {$\left[\begin{smallmatrix}f_i\\0\end{smallmatrix}\right]$}  (3');
\draw[->>, thin]   (3')  to node  {$\left[\begin{smallmatrix}g_{i} & 0\\0 & 1\end{smallmatrix}\right]$}  (4');
\draw[->, thin]   (2)  to node  {$1$}  (2');
\draw[->, thin]   (3)  to node  {$\left[\begin{smallmatrix}1\\0\end{smallmatrix}\right]$}  (3');
\draw[->, thin]   (4)  to node  {$\left[\begin{smallmatrix}1\\0\end{smallmatrix}\right]$}  (4');
\draw[right hook ->, thin]   (2'')  to node  {$a$}  (3'');
\draw[->>, thin]   (3'')  to node  {$b$}  (4'');
\draw[->, thin]   (2')  to node  {$1$}  (2'');
\draw[->, thin]   (3')  to node  {$u'_i$}  (3'');
\draw[->, thin]   (4')  to node  {$u_i$}  (4'');
    \end{tikzpicture} \]
\end{minipage} \\
 where $B'_{i}:=\bigoplus_{j\in X-\{i\}}B_{i}$. Now, by the universal
property of coproducts, there is $\gamma\in\Hom[][\bigoplus_{i\in X}E_{i}][D]$
such that $\gamma\circ\mu_{i}^{E}=u'_{i}\circ\left[\begin{smallmatrix}1\\
0
\end{smallmatrix}\right]$, for all $i\in X$. Moreover, since 
\[
\gamma\circ(\bigoplus_{i\in X}f_{i})\circ\mu_{i}^{A}=\gamma\circ\mu_{i}^{E}\circ f_{i}=u'_{i}\circ\left[\begin{smallmatrix}1\\
0
\end{smallmatrix}\right]\circ f_{i}=a=a\circ\nabla_{X}^{A}\circ\mu_{i}^{A},\:\forall i\in X
\]
we have that $\gamma\circ(\bigoplus_{i\in X}f_{i})=a\circ\nabla_{X}^{A}$.
And hence, by the universal property of pushouts between $\nabla_{X}^{A}$
and $\bigoplus_{i\in X}f_{i}$, there is a morphism $\gamma':Z_{H}\rightarrow D$
such that $\gamma'\circ f_{H}=a$. Therefore, since $a$ is a monomorphism,
$f_{H}$ is a monomorphism. This proves that $\mathcal{A}$ is Ab4
by Lemma \ref{lem:la inversa} and Theorem \ref{thm:Ab4 vs ext}. 
\end{proof}

\section{Injective and projective effacements\label{sec:eff}}

Injective (resp. projective) effacements are a notion that arose from
abstracting injective (resp. projective) objects \cite[Section 1.10]{Ab}.
Among their  applications we can mention that they are used to characterize
Ab4 categories in a similar way to what we have done in the previous
section (see \cite[Theorem A]{barr1975existence}, \cite[Section 1.10, Remark 1]{Ab},
\cite[Section 5, Example A]{leroux1971structures}, \cite[Corollary 1.4]{roos2006derived},
\cite[Theorem 1]{roos1966produits}). Which leads us to ask whether
there is a relationship between universal extensions and injective
effacements. The aim of this section will be to explore this question. 

\newcommandx\Coker[1][usedefault, addprefix=\global, 1=M]{\operatorname{Coker}\left(#1\right)}%

\begin{defn}
\cite[Definition 3.1]{rogalski2019well} Let $\mathcal{A}$ be an
abelian category, $A$ be an object of $\mathcal{A}$, and $\mathcal{S}$
be a class of objects in $\mathcal{A}$. An \textbf{$\mathcal{S}$-injective
effacement} of $A$ is a monomorphism $\iota:A\rightarrow\overline{A}$
such that for any monomorphism $f:A\rightarrow B$ with $\Coker[f]\in\mathcal{S}$,
there is a morphism $j:B\rightarrow\overline{A}$ such that $j\circ f=\iota$.
An \textbf{$\mathcal{S}$-projective effacement} is an epimorphism
satisfying the dual condition. Moreover, we say that an $\A$-injective
(resp. $\A$-projective) effacement is an \textbf{injective} (resp.
\textbf{projective}) \textbf{effacement}. 
\end{defn}

\begin{rem}
\label{rem: equivalencias injective effacement}Observe that a monomorphism
$\iota:A\rightarrow\overline{A}$ is an \textbf{$\mathcal{S}$-injective
effacement} if, and only if, $\Ext[N][1][\mathcal{A}][\iota]:\Ext[N][1][\mathcal{A}][A]\rightarrow\Ext[N][1][\mathcal{A}][\overline{A}]$
is the zero morphism $\forall N\in\mathcal{S}$(see \cite[Corollary 1.4]{roos2006derived}).
And hence, the following statements are equivalent:
\begin{enumerate}
\item we have an exact sequence $\suc[A][\overline{A}][C][\iota][\pi]$,
where $\iota$ is an $\mathcal{S}$-injective effacement, 
\item $\Ext[B][1][][\iota]:\Ext[B][1][][A]\rightarrow\Ext[B][1][][\overline{A}]$
is the zero morphism $\forall B\in\mathcal{S}$, 
\item $\Ext[B][1][][\pi]:\Ext[B][1][][\overline{A}]\rightarrow\Ext[B][1][][C]$
is injective $\forall B\in\mathcal{S}$, 
\item the connection morphism $\delta:\Hom[][B][C]\rightarrow\Ext[B][1][][A]$
is surjective $\forall B\in\mathcal{S}$. 
\end{enumerate}
In particular, an exact sequence $\suc[A][\overline{A}][B^{(X)}][\iota]$
is a universal extension if, and only if, $\iota$ is an $\{B\}$-injective
effacement. Therefore, we have that universal extensions are a special
(and easy-to-use) kind of injective effacements. 
\end{rem}

In the rest of the subsection we study projective effacements in Grothendieck
categories which are Ab4{*}. Concretely, we give a new proof of \cite[Corollary 1.4]{roos2006derived})
using universal co-extensions.
\begin{prop}
\label{prop:lambda-p2 vs proj eff} If $\G$ be a Grothendieck category
which is Ab4{*}, then every object of $\G$ has a projective effacement. 
\end{prop}

\begin{proof}
Notice first that in $\G$, there exists the universal co-extensions
for any couple of objects in $\G$ (see Theorem \ref{thm:Ab4 y Ext-peque sii ext universales}).
On the other hand, let $A$ be an object in $\G$. From \cite[Lemma 2.5.13, Proposition 2.5.16]{krause2021homological},
we know that there is $\kappa$ regular cardinal such that $A$ is
$\kappa$-presentable. From Remark \ref{rem:ordinal-cardinal-G} and
Proposition \ref{prop:pres implica 2 pres}, we know that there is
a regular cardinal $\lambda\geq\kappa$ such that $A$ is $\lambda$-FP2
and there is a set $\mathcal{S}$ of objects in $\G$ such that every
object in $\G$ is a $\lambda$-directed colimit of objects in $\mathcal{S}$.
Now, for each $S$ in $\mathcal{S}$, we take a universal co-extension
of $S$ by $A$, say $\overline{\epsilon_{S}}\in\Ext[A][1][\mathcal{A}][V_{S}]$,
where $V_{S}=S^{X_{S}}$ for some non-empty set $X_{S}$. Using the
Ab4{*} condition on $\G$ together with the Theorem \ref{thm:Ab4 vs ext},
we get that the canonical map 
$\Ext[A][1][\mathcal{A}][\prod_{S\in\mathcal{S}}V_{S}]\rightarrow\prod_{S\in\mathcal{S}}\Ext[A][1][\mathcal{A}][V_{S}]$
is bijective. And hence, there is an extension $\epsilon:\:\suc[\prod_{S\in\mathcal{S}}V_{S}][\overline{A}][A][\iota][\pi]$
such that $\overline{\pi_{S}^{V}\cdot\epsilon}=\overline{\epsilon_{S}}$,
for all $B\in\mathcal{S}$. We claim that $\pi:\overline{A}\twoheadrightarrow A$
is a projective effacement of $A$. Indeed, let $B$ be an object
in $\G$ and let $F:D\to\G$ be a $D$-diagram on $\G$ such that:
(i) $D$ is a $\lambda$-directed poset; (ii) $F(d)\in\mathcal{S}$,
for all $d\in D$; and (iii) $\mathsf{colim}_{D}F(d)=B$. Now, for
each $d$ in $D$ we consider an extension in $\G$ of the form: 
\[
\eta_{d}:\:\suc[F(d)][K_{d}][A][\iota_{d}][\pi_{d}]
\]
Thus, there is $f_{d}\in\Hom[\G][F(d)^{X_{F(d)}}][F(d)]$ such that
$\overline{f_{d}\cdot\epsilon_{F(d)}}=\overline{\eta_{d}}$. In particular,
we get $\overline{\eta_{d}}=\overline{(f_{d}\circ\pi_{F(d)}^{V})\cdot\epsilon}$.
This fact shows that the canonical morphism $\Ext[A][1][\G][F(d)]\to\Ext[\overline{A}][1][\G][F(d)]$
is zero. On the other hand, observe that we have the following commutative
diagram\\
\noindent\begin{minipage}[t]{1\columnwidth}%
\[ 
\begin{tikzpicture}[-,>=to,shorten >=1pt,auto,node distance=1.5cm,main node/.style=,x=3.5cm,y=1cm]
   \node (1)  at (0,0)   {$\operatorname{Ext}^{1}_{\mathcal{G}}(A,\mathsf{colim}_DF(d))$};    
   \node (2)  at (1.5,0)          {$\operatorname{Ext}^{1}_{\mathcal{G}}(\overline{A},\mathsf{colim}_DF(d))$};
       \node (3)  at (0,1.5)          {$\mathsf{colim}_D\operatorname{Ext}^{1}_{\mathcal{G}}(A,F(d))$};    
   \node (4)  at (1.5,1.5)          {$\mathsf{colim}_D\operatorname{Ext}^{1}_{\mathcal{G}}(\overline{A},F(d))$}; 
  
\draw[->, thin]  (1)  to  node  {$$} (2);
\draw[<-, thin]  (2)  to  node  {$$} (4);
\draw[<-, thin]  (1)  to  node  {$$} (3);
\draw[->, thin]  (3)  to  node  {$$} (4);
\end{tikzpicture}
\]
\end{minipage}\\
 where the morphism on the left is an isomorphism since $A$ is $\lambda$-FP2.
Hence, we deduce that the canonical morphism $\Ext[A][1][\G][B]\to\Ext[\overline{A}][1][\G][B]$
is zero, since the morphism on the top of the diagram is null. 
\end{proof}
\begin{cor}
\cite[Corollary 1.4]{roos2006derived}\label{cor:roos} Let $\G$
be a Grothendieck category. Then, $\G$ is Ab4{*} if, and only if,
every object of $\G$ has a projective effacement. 
\end{cor}

\begin{proof}
From Proposition \ref{prop:lambda-p2 vs proj eff} and the dual of
Theorem \ref{thm:Ab4 y Ext-peque sii ext universales}, we only need
to check that there exists the universal co-extensions for any couple
of objects in $\G$. Indeed, let $A,B$ be objects in $\G$ and let
$\pi:\overline{A}\twoheadrightarrow A$ be a projective effacement
of $A$. Using the dual of Remark \ref{rem: equivalencias injective effacement}(d),
we deduce that $\Ext[A][1][\G][B]$ is a set. Therefore, we can find
a complete  set 
\[
\left\{ \eta_{i}:\:\suc[B][E_{i}][A][a_{i}][b_{i}]\right\} _{i\in I}
\]
of representatives of $\Ext[A][1][\G][B]$. Moreover, for each $i\in I$
there is $f_{i}\in\Hom[\G][\operatorname{Ker}(\pi)][B]$ such that
$\overline{f_{i}\cdot\epsilon}=\overline{\eta_{i}}$, where $\epsilon:=\:\suc[\operatorname{Ker}(\pi)][\overline{A}][A][\iota][\pi]$.
Now, it follows from the universal property of products that there
is a morphism $f:\operatorname{Ker}(\pi)\to B^{I}$ such that $\pi_{i}^{B}\circ f=f_{i}$,
for all $i\in I$. And hence, $\overline{\epsilon_{B}}:=\overline{f\cdot\epsilon}$
is a universal co-extension of $B$ by $A$ since $\overline{\pi_{i}^{B}\cdot\epsilon_{B}}=\overline{(\pi_{i}^{B}\circ f)\cdot\epsilon}=\overline{f_{i}\cdot\epsilon}=\overline{\eta_{i}}$,
for all $i\in I$. 
\end{proof}

\section{Ext-universal objects in Ab3 abelian categories\label{Ab3}}

As we have seen in the subsection \ref{sub.Ab4}, if $\A$ is an Ab3
abelian category which is not Ab4, then there are objects $B$ and
$A$ in $\A$ such that there is no universal extension of $B$ by
$A$. The goal of this section is characterize those objects $V$
such that there is always a universal extension of $V$ by any other
object in $\mathcal{A}$. 
\begin{defn}
\cite[Definition 5.6]{parra2021tilting}\label{def:extuniversal}
Let $\mathcal{\mathcal{A}}$ be an abelian category and $A\in\mathcal{A}$.
\begin{enumerate}
\item An object $V$ in $\A$ is said to be \textbf{$\operatorname{Ext}^{1}$-universal
by $A$} when a universal extension of $V$ by $A$ exists in $\mathcal{A}$.
The class of objects that satisfy being $\operatorname{Ext}^{1}$-universal
by $A$ will be denoted by $\Extu_{\mathcal{A}}(A)$. 
\item An object $V$ is said to be \textbf{$\operatorname{Ext}^{1}$-universal}
if $V\in\bigcap_{A\in\mathcal{A}}\Extu_{\mathcal{A}}(A)$. The class
of $\operatorname{Ext}^{1}$-universal objects in $\A$ will be denoted
by $\Extu_{\mathcal{A}}.$ 
\end{enumerate}
\end{defn}

\begin{rem}
Let $\mathcal{A}$ be an $\operatorname{Ext}$-small Ab3 abelian category
and let $V\in\mathcal{A}$ be an $\operatorname{Ext}^{1}$-universal
object. It follows from Corollary \ref{cor:extension universal canonica}
that, for every $A\in\mathcal{A}$, we can consider the canonical
universal extension of $V$ by $A$. 
\end{rem}

\begin{lem}
\label{lem.sufficientcondi} Let $\mathcal{A}$ be an Ab3 abelian
category, and $A,V$ be objects in $\mathcal{A}$ such that $X:=\Ext[V][1][][A]$
is a set. Consider the following statements:
\begin{enumerate}
\item $V\in\Extu_{\mathcal{A}}(A^{(X)})$; 
\item if $\{\eta_{x}:\:\suc[A][E_{x}][V][f_{x}][g_{x}]\}_{x\in X}$ is a
complete set of representatives of $\Ext[V][1][][A]$, then $\bigoplus_{x\in X}f_{x}:A^{(X)}\rightarrow\bigoplus_{x\in X}E_{x}$
is monic, and hence $\bigoplus_{x\in X}\eta_{x}$ is a short exact
sequence; 
\item $V\in\Extu_{\mathcal{A}}(A)$. 
\end{enumerate}
Then, $(a)\Rightarrow(b)$ and $(b)\Rightarrow(c)$. 
\end{lem}

\begin{proof}
$(b)\Rightarrow(c)$ It follows straightforward from Theorem \ref{teo:secondmain},
\cite[Lemma 4.2]{argudin2022exactness}, Lemma \ref{lem:la extension colimite mediante la codiagonal },
and Lemma \ref{lem:phi es biyectiva}.

\noindent %
\begin{minipage}[t]{0.5\columnwidth}%
\indent $(a)\Rightarrow(b)$ Consider the set of extensions $\left\{ \overline{\mu_{x}^{A}\cdot\eta_{x}}\right\} _{x\in X}$.
By Lemma \ref{lem:phi es biyectiva} there is an extension $\overline{\rho}\in\Ext[V^{(X)}][1][][A^{(X)}]$
such that $\overline{\rho\cdot\mu_{x}^{V}}=\overline{\mu_{x}^{A}\cdot\eta_{x}}$,
for all $x\in X$. As a consequence of this, we have the commutative
diagram on the right, where $\rho:\:\suc[A^{(X)}][W][V^{(X)}][f][g]$.
Now, by the universal property of coproducts, we get that there is
$\gamma\in\Hom[][\bigoplus_{x\in X}E_{x}][W]$ such that $\gamma\circ\mu_{x}^{E}=\mu''_{x}\circ\mu'_{x}$,
for all $x\in X$. Observe that $\gamma\circ(\bigoplus_{x\in X}f_{x})=f$.
And hence, since $f$ is monic, $\bigoplus_{x\in X}f_{x}$ is a monomorphism.%
\end{minipage}\hfill{}%
\fbox{\begin{minipage}[t]{0.45\columnwidth}%
\[ 
\begin{tikzpicture}[-,>=to,shorten >=1pt,auto,node distance=1.5cm,main node/.style=,x=1.5cm,y=1.5cm]

\node[main node][] (A) at (-1.2,-2.5)    {$A^{(X)}$};
   \node[main node][] (B) at (0.2,-2.5)     {$\bigoplus _{x\in X} E_x $};

   \node[main node][] (2) at (0,0)      {$A$};

       \node[main node][] (3) [right of=2]  {$E_x$};

       \node[main node][] (4) [right of=3]  {$V$};

  \node[main node][] (2') [below of=2]  {$A^{(X)}$};

      \node[main node][] (3') [right of=2']  {$W_x$};       

\node[main node][] (4') [right of=3']  {$V$};

\node[main node][] (2'') [below of=2']  {$A^{(X)}$};

   \node[main node][] (3'') [right of=2'']  {$W$};

       \node[main node][] (4'') [right of=3'']  {$V^{(X)}$};    

\draw[right hook->, thin]   (2)  to [below] node  {$f_x$}  (3);

\draw[->>, thin]   (3)  to [below] node  {$g_x$}  (4);

\draw[right hook->, thin]   (2')  to node  {$f'_x$}  (3');
\draw[->>, thin]   (3')  to node  {$g'_x$}  (4');
\draw[right hook->, thin]   (2'')  to node  {$f$}  (3'');
\draw[->>, thin]   (3'')  to node  {$g$}  (4'');
\draw[right hook->, thin]   (2)  to node  {$\mu ^{A} _x$}  (2');
\draw[right hook->, thin]   (3)  to node  {$\mu '_x$}  (3');
\draw[-, double]   (4)  to node  {$$}  (4');
\draw[-, double]   (2')  to node  {$$}  (2'');
\draw[right hook->, thin]   (3')  to node  {$\mu ''_x$}  (3'');
\draw[right hook->, thin]   (4')  to [left] node  {$\mu ^{V}_x$}  (4'');
\draw[->, thin]   (2)  to [above left] node  {$\mu ^A _x$}  (A);
\draw[->, thin]   (3)  to [left] node  {$\mu ^E _x$}  (B);
\draw[<-, dashed]   (2'')  to [above left] node  {$1_A$}  (A);
\draw[<-, dashed]   (3'')  to node  {$\gamma$}  (B);
\draw[->, thin]   (A)  to [below] node  {$\bigoplus _{x\in X}f_x$}  (B);    
\end{tikzpicture}
\]
\end{minipage}}
\end{proof}
Let us study how to find new $\operatorname{Ext}^{1}$-universal objects
from old ones. 
\begin{lem}
Let $\mathcal{A}$ be an Ab3 abelian category, $\{V_{i}\}_{i\in I}$
be a set of objects in $\mathcal{A}$ and $A\in\mathcal{A}$. Then,
the following statements hold true (here $V:=\bigoplus_{i\in I}V_{i}$
and $X:=\Ext[V][1][][A]$):
\begin{enumerate}
\item If there is a universal extension of $V_{i}$ by $A$, for all $i\in I$,
then $\Ext[V^{(S)}][1][][A]$ is a set, for all set $S$. In particular,
$X$ is a set. 
\item If $V_{i}\in\Extu_{\mathcal{A}}(A)$, for all $i\in I$ and the natural
maps {\small{}{} 
\[
\Ext[V^{(X)}][1][][A]\rightarrow\prod_{i\in I}\Ext[V_{i}^{(X)}][1][][A]\mbox{ and }\Ext[V][1][][A]\rightarrow\prod_{i\in I}\Ext[V_{i}][1][][A]
\]
} are bijective, then $V\in\Extu_{\mathcal{A}}(A)$. 
\end{enumerate}
\end{lem}

\begin{proof}
Before proceeding with the proof, let us recall the following fact.
Let $S$ be a set. Consider the set $\{V_{i,s}\}_{(i,s)\in I\times S}\subseteq\mathcal{A}$,
where $V_{i,s}:=V_{i}$, for all $s\in S$, and the coproduct $V':=\bigoplus_{(i,s)\in I\times S}V_{i,s}$
with the canonical inclusions $\alpha_{i,s}:V_{i,s}\rightarrow V'$.
In one hand, we can see that, for $i\in I$ fixed, there is a unique
morphism $u_{i}:V_{i}^{(S)}\rightarrow V'$ such that $u_{i}\circ v_{i,s}=\alpha_{i,s}$,
for all $s\in S$, where $v_{i,s}:V_{i}\rightarrow V_{i}^{(S)}$ is
the $s$-th canonical inclusion. Moreover, it can be shown that $V'$
is the coproduct $\bigoplus_{i\in I}V_{i}^{(S)}$ and that $\{u_{i}:V_{i}^{(S)}\rightarrow V'\}_{i\in I}$
is the set of canonical inclusions. On the other hand, we can see
that, for $s\in S$ fixed, there is a unique morphism $w_{s}:\bigoplus_{i\in I}V_{i,s}\rightarrow V'$
such that $w_{s}\circ\mu_{i}^{V}=\alpha_{i,s}$, for all $i\in I$,
where $\mu_{i}^{V}:V_{i}\rightarrow\bigoplus_{i\in I}V_{i}$ is the
$i$-th canonical inclusion. Moreover, it can also be shown that $V'$
is the coproduct $\left(\bigoplus_{i\in I}V_{i}\right)^{(S)}$ and
that $\{w_{s}:V_{i}^{(X)}\rightarrow V'\}_{s\in S}$ is the set of
canonical inclusions. In other words, $\bigoplus_{i\in I}\left(V_{i}^{\left(S\right)}\right)=\left(\bigoplus_{i\in I}V_{i}\right)^{(S)}$.
We will use the notation presented in this paragraph throughout the
proof for $S=X$.
\begin{enumerate}
\item Let $S$ be a set. Observe that, for each $i\in I$, $\Ext[V_{i}][1][][A]$
is a set by Definition \ref{def:extension universal}(c). Thus, $\Ext[V_{i}^{(S)}][1][][A]$
also is a set by Lemma \ref{lem:phi es biyectiva}. Moreover, since
there is an injective function 
\[
\Ext[\bigoplus_{i\in I}V_{i}^{(S)}][1][][A]\rightarrow\prod_{i\in I}\Ext[V_{i}^{(S)}][1][][A]
\]
(see \cite[Lemma 4.2]{argudin2022exactness}), $\Ext[\bigoplus_{i\in I}V_{i}^{(S)}][1][][A]$
is a set. 
\item We have the following chain of isomorphisms. 
\begin{alignat*}{1}
\Ext[\left(\bigoplus_{i\in I}V_{i}\right)^{(X)}][1][][A] & \overset{(1)}{\cong}\Ext[\bigoplus_{i\in I}\left(V_{i}^{(X)}\right)][1][][A]\\
 & \overset{(2)}{\cong}\prod_{i\in I}\Ext[V_{i}^{(X)}][1][][A]\\
 & \overset{(3)}{\cong}\prod_{i\in I}\Ext[V_{i}][1][][A]^{X}\\
 & \overset{(4)}{\cong}\Ext[\bigoplus_{i\in I}V_{i}][1][][A]^{X}\mbox{.}
\end{alignat*}
Indeed, (1) can be seen as a equality by the arguments above; (2)
and (4) are our hypotheses using the fact that 
\[
\prod_{i\in I}\Ext[V_{i}][1][][A]^{X}=\left(\prod_{i\in I}\Ext[V_{i}][1][][A]\right)^{X}=\prod_{i\in I}\left(\Ext[V_{i}][1][][A]^{X}\right);
\]
and (3) follows from Lemma \ref{lem:phi es biyectiva}. Observe that
the correspondence rule of this isomorphism is given by 
\[
\overline{\eta}\overset{(1)}{\mapsto}\overline{\eta}\overset{(2)}{\mapsto}\left(\overline{\eta\cdot u_{i}}\right)_{i\in I}\overset{(3)}{\mapsto}\left(\overline{\eta\cdot(u_{i}\circ v_{i,x})}\right)_{i\in I,x\in X}\overset{(4)}{\mapsto}\left(\overline{\eta_{x}}\right)_{x\in X}\mbox{,}
\]
where $\overline{\eta_{x}}$ is an extension such that $\overline{\eta_{x}\cdot\mu_{i}^{V}}=\overline{\eta\cdot(u_{i}\circ v_{i,x})}$,
for all $i\in I$. Finally, note that $w_{x}\circ\mu_{i}^{V}=\alpha_{i,x}=u_{i}\circ v_{i,x}$,
for all $(i,x)\in I\times X$ by the above arguments. Hence, $\overline{\eta_{x}\cdot\mu_{i}^{V}}=\overline{\eta\cdot(w_{x}\circ\mu_{i}^{V})}$,
for all $\left(i,x\right)\in I\times X$. This means that $\overline{\eta_{x}}=\overline{\eta\cdot w_{x}}$
because $\Ext[\mu_{i}^{V}][1][][A]$ is injective. Therefore, the
correspondence rule of the chain of isomorphisms is $\overline{\eta}\mapsto\overline{\eta\cdot w_{x}}$.
This means that the natural map $\Ext[V^{(X)}][1][][A]\rightarrow\Ext[V][1][][A]^{X}$
is bijective. And hence, by Corollary \ref{cor:caract ext univ},
there is a universal extension of $V$ by $A$. 
\end{enumerate}
\end{proof}
Let $\mathcal{A}$ be an Ab3 abelian category and $A\in\mathcal{A}$.
In general the class $\Extu_{\mathcal{A}}(A)$ is not closed under
arbitrary coproducts. An example of this can be seen in Corollary
\ref{exa:contraejemplo }.
\begin{cor}
\label{cor:coproducto de universales}Let $\mathcal{A}$ be an Ab3
abelian category. Then, the following statements hold true:
\begin{enumerate}
\item The class $\Extu_{\mathcal{A}}$ is closed under finite coproducts. 
\item If $\{V_{i}\}_{i\in I}$ is a set in $\Extu_{\mathcal{A}}$ and the
natural map 
\[
\Ext[\bigoplus_{i\in I}\left(V_{i}^{(X)}\right)][1][][A]\rightarrow\prod_{i\in I}\Ext[V_{i}^{(X)}][1][][A]
\]
is bijective for every set $X$ and every $a\in\mathcal{A}$, then
$\bigoplus_{i\in I}V_{i}\in\Extu_{\mathcal{A}}$. 
\end{enumerate}
\end{cor}


\begin{lem}
Let $\mathcal{A}$ be an Ab3 abelian category, $\omega:\:\suc[A][E][V^{(X)}][f][g]$
be a universal extension, $\eta:\:\suc[U][V][W][a][b]$ be an exact
sequence, and $(g':E'\rightarrow U^{(X)},e:E'\rightarrow E)$ be the
pull-back of $g$ and $a^{(X)}$. If any of the following conditions
hold true, then $\omega\cdot a^{(X)}$ is a universal extension of
$U$ by $A$.
\begin{enumerate}
\item $\Ext[V][1][][e]$ is injective and $\Ext[a][1][][A]$ is surjective. 
\item $\Ext[U][1][][e]$ is injective and $\Ext[a][1][][A]$ is surjective. 
\item If $\omega$ is the canonical universal extension and $a$ is a split
monomorphism. 
\end{enumerate}
\end{lem}

\begin{proof}
It is a known fact that, in this scenario, we get the following commutative
diagram with exact rows. \\
\noindent\begin{minipage}[t]{1\columnwidth}%
\[ \begin{tikzpicture}[-,>=to,shorten >=1pt,auto,node distance=1.5cm,main node/.style=,x=1.5cm,y=1.5cm]
   \node[main node] (A) at (0,0)      {$\omega \cdot a^{(X)} :$};  
     \node[main node] (1) at (0,0)      {$$};      
 \node[main node] (2) [right of=1]  {$A$};     
  \node[main node] (3) [right of=2]  {$E'$};  
     \node[main node] (4) [right of=3]  {$U ^{(X)}$}; 
      \node[main node] (5) [right of=4]  {$$};

   \node[main node] (A) at (0,-1)      {$\omega :$}; 
      \node[main node] (1') [below of=1]      {$$}; 
      \node[main node] (2') [right of=1']  {$A$};  
     \node[main node] (3') [right of=2']  {$E$};  
     \node[main node] (4') [right of=3']  {$V ^{(X)}$}; 
      \node[main node] (5') [right of=4']  {$$};

\draw[right hook ->, thin]   (2)  to [below] node  {$f'$}  (3);
\draw[->>, thin]   (3)  to [below] node  {$g'$}  (4);
\draw[right hook ->, thin]   (2')  to node  {$f$}  (3');
\draw[->>, thin]   (3')  to node  {$g$}  (4');
\draw[-, double]   (2)  to node  {$$}  (2');
\draw[->, thin]   (3)  to node  {$e$}  (3');
\draw[->, thin]   (4)  to node  {$a^{(X)}$}  (4');
    \end{tikzpicture} \]
\end{minipage}\\
 Consider the diagram below, where $^{k}(?,-):=\Ext[?][k][\A][-]$,
for all $k\in\{1,2\}$. Observe that it is commutative since it is
obtained through the left square in the diagram above. Moreover, the
leftmost column is exact by the long exact sequence of homology induced
by $\eta$. Now, since $\omega$ is a universal extension, we have
that $\Ext[V][1][][f]=0$ by Definition \ref{def:extension universal}(a).
Let us prove that $\Ext[U][1][][f']=0$ in each case to conclude that
$\omega\cdot a^{(X)}$ is a universal extension.
\begin{enumerate}
\item If $\Ext[V][1][][e]$ is injective, then $\Ext[V][1][][f']=0$ since
$\Ext[V][1][][e]\circ\Ext[V][1][][f']=\Ext[V][1][][f]=0$. And thus,
\[
\Ext[U][1][][f']\circ\Ext[a][1][][A]=\Ext[a][1][][E']\circ\Ext[V][1][][f']=0\mbox{.}
\]
Hence, if in addition $\Ext[a][1][][A]$ is surjective, then $\Ext[U][1][][f']$
is the zero map.
\item[(b)] It is proved in a similar way as (a).
\end{enumerate}
\noindent %
\begin{minipage}[t]{0.4\columnwidth}%
\begin{enumerate}
\item[(c)] Since $a$ is a split monomorphism, there is $a'\in\Hom[][V][U]$
such that $a'a=1_{U}$. Let $\overline{\eta}\in\Ext[U][1][][A]$.
Observe that $\overline{\eta}\cdot a'\in\Ext[V][1][][A]$. And hence,
there is $i\in X$ such that $\overline{\omega}\cdot\mu_{i}^{V}=\overline{\eta}\cdot a'$.
Therefore, $\overline{\eta}=\overline{\eta\cdot(a'\circ a)}=\overline{\omega\cdot(\mu_{i}^{V}\circ a)}=\overline{\omega\cdot(a^{(X)}\circ\mu_{i}^{U})}$.
As a result, we have that $\Ext[U][1][][f']$ is the zero map.
\end{enumerate}
\end{minipage}\hfill{}%
\fbox{\begin{minipage}[t]{0.55\columnwidth}%
\[ \begin{tikzpicture}[-,>=to,shorten >=1pt,auto,node distance=3cm,main node/.style=,x=1.5cm,y=1.5cm]
   \node[main node] (1) at (0,0)      {$^1 (V,A)$};
       \node[main node] (2) [right of=1]  {$^1 (V,E')$};
       \node[main node] (3) at (330:2cm)  {$^1 (V,A)$};
       \node[main node] (4) [right of=3]  {$^1 (V,E)$};
       \node[main node] (5) [below of=1]  {$^1 (U,A)$};
       \node[main node] (6) [right of=5]  {$^1 (U,E')$};
       \node[main node] (7) [below of=3]  {$^1 (U,A)$};
       \node[main node] (8) [below of=4]  {$^1 (U,E)$};       
\node[main node] (9) at (0,-3)     {$^2 (W,A)$};

\draw[->, thin]   (1)  to node  {$\scriptstyle{ ^1 (V,f')}$}  (2);

\draw[-, double]   (1)  to node  {$$}  (3);

\draw[->, thin]   (1)  to node  {$\scriptstyle{ ^1 (a,A)}$}  (5);
\draw[->, thin]   (3)  to node  {$\scriptstyle{ ^1 (V,f)}$}  (4);
\draw[->, thin]   (3)  to node  {$$}  (7);

\draw[->, thin]   (2)  to node  {$\scriptstyle{ ^1 (V,e)}$}  (4);

\draw[->, thin]   (2)  to node  {$\scriptstyle{ ^1 (a,E')}$}  (6);

\draw[->, thin]   (4)  to node  {$\scriptstyle{ ^1 (a,E)}$}  (8);

\draw[->, thin]   (5)  to node  {$\scriptstyle{ ^1 (U,f')}$}  (6);

\draw[-, double]   (5)  to node  {$$}  (7);

\draw[->, thin]   (7)  to node  {$\scriptstyle{ ^1 (U,f)}$}  (8);

\draw[->, thin]   (6)  to node  {$\scriptstyle{ ^1 (U,e)}$}  (8);

\draw[->, thin]   (5)  to node  {$$}  (9);

    \end{tikzpicture}
\]
\end{minipage}}

\end{proof}
\begin{cor}
\textup{\label{cor:sumandodirectodeext-universal} Let $\mathcal{A}$}
be an Ab3 abelian category and $V$ an $\operatorname{Ext}^{1}$-universal
object. Then, every direct summand of $V$ is a $\operatorname{Ext}^{1}$-universal
object. 
\end{cor}

\section{Co-$\operatorname{Ext}^{1}$-universal torsion groups}

\label{torsion} Throughout this section, $p$ is a prime number and
$\mathcal{T}$ is either $\mathcal{T}_{Z}$ or $\mathcal{T}_{p}$
(see Example \ref{exa:tor-her-classic}). We recall that $\mathcal{T}$
is a Grothendieck category which is not Ab4{*} (see Remark \ref{rem:Tor-her-Gro}
and Examples \ref{exa:tor-her-classic},\ref{exam:no-Ab4*}). Thus,
the abelian category $\mathcal{T}^{op}$ has objects which are not
Ext$^{1}$-universal. In the sequel, we characterize those abelian
groups in $\mathcal{T}$ which are Ext$^{1}$-universal in $\mathcal{T}^{op}$.
For this task, we will study the dual definition of universal extension
(see Definition \ref{def:extension co-universal}) and Ext$^{1}$-universal
object.


\begin{defn}
An abelian group $G$ in $\mathcal{T}$ is said to be \textbf{co-$\operatorname{Ext}^{1}$-universal}
in $\mathcal{T}$ when a universal co-extension of $G$ by any other
abelian group in $\mathcal{T}$ exists in $\mathcal{T}$. 
\end{defn}

We recall that $\mathcal{T}_{Z}$ (resp. $\mathcal{T}_{p}$) is the
torsion class of a hereditary torsion pair in $\Ab$ (see Example
\ref{exa:tor-her-classic}(a)). We will denote by $t:\Ab\to\Ab$ (resp.
$t_{p}:\Ab\to\Ab$) the associated torsion radical. In this case,
the functor $t$ (resp. $t_{p}$) is left exact. Moreover, let $\overline{t}$
denote either $t$ or $t_{p}$ in this section.
\begin{rem}
\label{rem:propertiesoft} Note that for each extension $\eta:\;\suc[][][][f][g]$
in $\Ab$, we obtain an exact sequence $\overline{t}(\eta):\;\overline{t}(N)\overset{\overline{t}(f)}{\hookrightarrow}\overline{t}(M)\overset{\overline{t}(g)}{\rightarrow}\overline{t}(K)$
in $\Ab$, where $\overline{t}(g)$ is not necessarily an epimorphism.
On the other hand, we recall that the product in $\mathcal{T}$ of
a family of abelian groups $(G_{i})_{i\in I}\subseteq\mathcal{T}$
is given by $\overline{t}(\prod_{i\in I}G_{i})$, where $\prod_{i\in I}G_{i}$
is the product of such family in $\Ab$, and the $i$-th canonical
projections $\pi_{i}^{\overline{t}G}:=\pi_{i}^{G}\circ\iota_{\prod_{i\in I}G_{i}}:\overline{t}\left(\prod_{i\in I}G_{i}\right)\rightarrow G_{i}$,
where $\iota_{\prod_{i\in I}G}:\overline{t}(\prod_{i\in I}G_{i})\rightarrow\prod_{i\in I}G_{i}$
is the canonical inclusion. 
\end{rem}

The following lemma establishes the relationship between the universal
co-extensions in $\Ab$ and the universal co-extensions in $\mathcal{T}$. 
\begin{lem}
\label{lem:la canonica de torsion}Let $A,B\in\mathcal{T}$ such that
there is a universal co-extension of $B$ by $A$ in $\mathcal{T}$.
If $\eta:\:\suc[B^{X}][E][A][f][g]$ is the canonical universal co-extension
of $B$ by $A$ in $\Ab$, then $\overline{t}(\eta)$ is a short exact
sequence and it is the canonical universal co-extension of $B$ by
$A$ in $\mathcal{T}$. 
\end{lem}

\begin{proof}
Let $H=\left\{ \eta_{i}:\:\suc[B][E_{i}][A][f_{i}][g_{i}]\right\} _{i\in X}$
be a complete set of representatives of $\Ext[A][1][\Ab][B]$. By
the dual of Corollary \ref{cor:extension universal canonica} together
with the Remark \ref{rem:propertiesoft}, we know that the canonical
universal co-extension of $B$ by $A$ in $\mathcal{T}$ is the exact
sequence $\eta':\:\suc[\overline{t}\left(B^{X}\right)][E'][A][f'][g']$
built through the pullback of $\overline{t}\left(\Delta_{X}^{A}\right)$
and $\overline{t}\left(\prod_{i\in X}g_{i}\right)$ (the dual construction
described in Lemma \ref{lem:colim}). Similarly, we know that $\eta$
is built through the pullback of $\Delta_{X}^{A}$ and $\prod_{i\in X}g_{i}$.
Therefore, we have the following commutative diagram.\\
\noindent\fbox{\begin{minipage}[t]{1\columnwidth - 2\fboxsep - 2\fboxrule}%
\[ \begin{tikzpicture}[-,>=to,shorten >=1pt,auto,node distance=2.9cm,main node/.style=,x=2.9cm,y=2.9cm]
\node[main node] (E) at (-.5,1)      {$\overline{t}(\eta):$};
\node[main node] (E') at (-.5,0)      {$ \eta ':$};
\node[main node] (E'') at (-.5,-1)      {$\overline{t}( \prod \eta _i ):$};
\node[main node] (E''') at (-.5,-1)      {$\overline{t}( \prod \eta _i ):$};
   \node[main node] (1) at (0,1)       {$\overline{t}(B^X)$};  
     \node[main node] (2) [right of=1]  {$\overline{t}(E)$};    
   \node[main node] (3) [right of=2]  {$A$};    
   \node[main node] (1') at (0,0)       {$\overline{t}(B^X)$}; 
      \node[main node] (2') [right of=1']  {$E'$};   
    \node[main node] (3') [right of=2']  {$A$};    
   \node[main node] (1'') at (0,-1)         {$\overline{t}(B^X)$}; 
      \node[main node] (2'') [right of=1'']  {$\overline{t}(\prod E_i)$};
       \node[main node] (3'') [right of=2'']  {$\overline{t}(A^X)$}; 
  
   \node[main node] (1''') at (315:2.32cm)  {$B^X$};    
   \node[main node] (2''') [right of=1''']  {$E$};  
     \node[main node] (3''') [right of=2''']  {$A$};    
   \node[main node] (1'''') [below of=1''']  {$B^X$}; 
      \node[main node] (2'''') [right of=1'''']  {$\prod E_i$};  
     \node[main node] (3'''') [right of=2'''']  {$A^X$};    
\node[main node] (E) at (-.5,1)      {$\overline{t}( \eta):$};
\node[main node] (E') at (-.5,0)      {$ \eta ':$};
\node[main node] (E'') at (-.5,-1)      {$\overline{t}( \prod \eta _i ):$};
\node[main node] (E''') [right= .5cm of 3''']      {$:  \eta $};
\node[main node] (E'''') [right= .5cm of 3'''']      {$:  \prod \eta _i $};
\draw[right hook->, thin]   (1)  to [below] node  {$\scriptstyle \overline{t}(f)$}  (2);
\draw[->, thin]   (2)  to [below] node  {$\scriptstyle \overline{t}(g)$}  (3);
\draw[right hook->, thin]   (1')  to node  {$\scriptstyle f'$}  (2');
\draw[->>, thin]   (2')  to  node  {$\scriptstyle g'$}  (3');
\draw[right hook->, thin]   (1'')  to node  {$$}  (2'');
\draw[->, thin]   (2'')  to  node  {$$}  (3'');
\draw[right hook->, thin]   (1''')  to [below] node  {$\scriptstyle f$}  (2''');
\draw[->>, thin]   (2''')  to [below] node  {$\scriptstyle g$}  (3''');  
\draw[right hook->, thin]   (1'''')  to [below] node  {$\scriptstyle \prod f_i$}  (2'''');
\draw[->>, thin]   (2'''')  to [below] node  {$\scriptstyle \prod g_i$}  (3'''');
\draw[-, double]   (1)  to node  {$$}  (1');
\draw[-, double]   (1')  to  node  {$$}  (1'');
\draw[right hook->, thin]   (1')  to [] node  {$$}  (1''');
\draw[right hook->, thin]   (1'')  to [below left] node  {$\scriptstyle \iota _{B^X}$}  (1'''');
\draw[-, double]   (1''')  to  node   {$$}  (1'''');
\draw[right hook->, thin]   (1)  to  node   {$\scriptstyle \iota _{B^X}$}  (1''');
\draw[->, dashed]   (2')  to node  {$\scriptstyle \gamma '$}  (2);
\draw[->, dashed]   (2')  to  node  {$\scriptstyle \gamma$}  (2''');
\draw[right hook->, thin]   (2')  to [left] node  {$\scriptstyle \Delta ''$}  (2'');
\draw[right hook->, thin]   (2'')  to [below left] node  {$\scriptstyle \iota _{\prod E}$}  (2'''');
\draw[right hook->, thin]   (2''')  to  node   {$\scriptstyle \Delta '$}  (2'''');
\draw[right hook->, thin]   (2)  to  node   {$\scriptstyle \iota _E$}  (2''');
\draw[-, double]   (3')  to node          {$$}          (3);
\draw[-, double]   (3')  to  node  {$$}            (3''');
\draw[right hook->, thin]   (3')  to [left] node    {$\scriptstyle \overline{t}(\Delta ^{A}_X)$}          (3'');
\draw[right hook->, thin]   (3'')  to [below left] node   {$\scriptstyle \iota _{A ^{X}}$}  (3'''');
\draw[right hook->, thin]   (3''')  to  node   {$\scriptstyle \Delta^{A}_X$}  (3'''');
\draw[right hook->, thin]   (3)  to  node   {$$}  (3''');
    \end{tikzpicture} 
\]
\end{minipage}} \\
Let us prove that $\overline{t}(g)$ is surjective and that $\eta'$
is equivalent to $\overline{t}\left(\eta\right)$. For this, observe
that $\Delta_{X}^{A}\circ g'=(\prod_{i\in X}g_{i})\circ\iota_{\prod_{i\in X}E_{i}}\circ\Delta''$.
And hence, by the universal property of pullbacks applied to $\Delta_{X}^{A}$
and $\prod_{i\in X}g_{i}$, there is $\gamma\in\Hom[\Ab][E'][E]$
such that $\Delta'\circ\gamma=\iota_{\prod_{i\in X}E_{i}}\circ\Delta''$
and $g\circ\gamma=g'$. Moreover, since $\Delta'\circ\gamma\circ f'=\Delta'\circ f\circ\iota_{B^{X}}$
and $\Delta'$ is monic, we have that $\gamma\circ f'=f\circ\iota_{B^{X}}$.
Now, since $E'\in\mathcal{T}$, there is $\gamma'\in\Hom[\Ab][E'][t(E)]$
such that $\gamma=\iota_{E}\circ\gamma'$. Observe that $\gamma'\circ f'=\overline{t}(f)$
and $\overline{t}(g)\circ\gamma'=g'$ since $\iota_{E}\circ\gamma'\circ f'=\iota_{E}\circ\overline{t}(f)$,
$\Delta_{X}^{A}\circ\overline{t}(g)\circ\gamma'=\Delta_{X}^{A}\circ g'$
and $\iota_{E}$ and $\Delta_{X}^{A}$ are monic. Therefore, $\overline{t}(g)$
is surjective and that $\eta'$ is equivalent to $\overline{t}\left(\eta\right)$. 
\end{proof}
\begin{prop}
\label{prop:torsion-cotorsionesco-ext-universal}Let $M\in\T$ be
a cotorsion group (i.e. $\Ext[\mathbb{Q}/\mathbb{Z}][1][\Ab][M]=0$).
Then, $M$ is co-$\operatorname{Ext}^{1}$-universal in $\mathcal{T}$. 
\end{prop}

\begin{proof}
Since $M$ is a torsion cotorsion group, it follows from the Baer-Fomin
Theorem that $M=D\oplus B$, where $D$ is an injective abelian group
and $B$ is a bounded abelian group (see \cite[Chapter 15, Theorem 1.6]{fuchs2015abelian}).
By Remark \ref{rem:extensiones universales}(a), every injective torsion
group is co-$\operatorname{Ext}^{1}$-universal in $\mathcal{T}$. Hence,
by Corollary \ref{cor:coproducto de universales}(a), it is enough
to show that $B$ is co-$\operatorname{Ext}^{1}$-universal in $\mathcal{T}$.
Let $A\in\mathcal{T}$. Consider a universal co-extension of $B$
by $A$ in $\Ab$, say $\eta:\:\suc[B^{X}][E][A]$. Observe that $n\cdot B^{X}=0$, for some nonzero integer $n$.
Therefore, $B^{X},A\in\mathcal{T}$, and hence $E\in\mathcal{T}$
because $\mathcal{T}$ is closed under extensions in $\Ab$. Therefore,
$\eta$ is a universal co-extension of $B$ by $A$ in $\mathcal{T}$. 
\end{proof}
\begin{prop}
\label{lem:caracterizacion_existencia_coextuniv} The following statements
are equivalent for $T,S\in\T$:
\begin{enumerate}
\item there is a universal co-extension of $T$ by $S$ in $\mathcal{T}$;
\item $\Ext[S][1][\mathbb{\Ab}][\iota_{T^{X}}]:\Ext[S][1][\mathbb{\Ab}][\overline{t}(T^{X})]\rightarrow\Ext[S][1][\mathbb{\Ab}][T^{X}]$
is an isomorphism, for every set $X$; 
\item $\Ext[S][1][\mathbb{\Ab}][T^{X}/\overline{t}(T^{X})]=0$, for every
set $X$. 
\end{enumerate}
\end{prop}

\begin{proof}
Let $X$ be a set and consider the  exact sequence $\suc[\overline{t}(T^{X})][T^{X}][T^{X}/\overline{t}(T^{X})][\iota_{T^{X}}][\pi_{T^{X}}]$.
Applying $\Hom[\mathbb{\Ab}][S][-]$, we get the exact sequence 
\[
0\rightarrow{}^{1}(S,t(T^{X}))\overset{^{1}(S,\iota_{T^{X}})}{\longrightarrow}{}^{1}(S,T^{X})\overset{}{\longrightarrow}{}^{1}(S,T^{X}/\overline{t}(T^{X}))\rightarrow0
\]
where the last term is zero because $\Ab$ is hereditary, and the
first term is zero because $S\in\mathcal{T}$ and $T^{X}/\overline{t}(T^{X})$
is torsion-free (with respect to $\mathcal{T}$). Now, observe $\Ext[S][1][\Ab][\overline{t}(T^{X})]=\Ext[S][1][\mathcal{T}][\overline{t}(T^{X})]$,
$\Ext[S][1][\Ab][T]^{X}=\Ext[S][1][\mathcal{T}][T]^{X}$, and that
the morphism 
\[
\Phi_{\mathcal{T}}:\Ext[S][1][\mathcal{T}][\overline{t}(T^{X})]\rightarrow\Ext[S][1][\mathcal{T}][T]^{X}\text{, }\overline{\eta}\mapsto\left(\overline{\left(\pi_{i}\circ\iota_{T^{X}}\right)\cdot\eta}\right)_{i\in X}\text{,}
\]
 is factored as $\Phi_{\mathcal{T}}=\Phi\circ\Ext[S][1][\Ab][\iota_{T^{X}}]$,
where $\Phi:\Ext[S][1][\Ab][T^{X}]\rightarrow\Ext[S][1][\Ab][T]^{X}$
is the isomorphism defined as $\overline{\eta}\mapsto\left(\overline{\pi_{i}\cdot\eta}\right)_{i\in X}$
(see Theorem \ref{thm:Ab4 vs ext}). Hence, the following statements
are equivalent: (i) $\Phi_{\mathcal{T}}$ is an isomorphism, (ii)
$\Ext[S][1][\mathcal{T}][\iota_{T^{X}}]$ is an isomorphism, and (iii)
$\Ext[S][1][\Ab][T^{X}/\overline{t}(T^{X})]=0$. Now, it follows from
the dual results of Lemma \ref{lem:phi es biyectiva} and Corollary
\ref{cor:caract ext univ} that (i) holds for every set $X$ if and
only if (a) holds true. Therefore, statements (a), (b), and (c) are
equivalent.
\end{proof}
\begin{prop}
\label{prop:VX cociente torsion es divisible} The following statements
are equivalent for $T\in\T$:
\begin{enumerate}
\item $T$ is co-$\operatorname{Ext}^{1}$-universal in $\mathcal{T}$;
\item $T^{X}/\overline{t}(T^{X})$ is injective in $\Ab$, for every set
$X$;
\item every quotient of $T$ is co-$\operatorname{Ext}^{1}$-universal in
$\mathcal{T}$. 
\end{enumerate}
\end{prop}

\begin{proof}
Observe that $(c)\Rightarrow(a)$ is trivial and that $(b)\Rightarrow(a)$
follows straightforward from Lemma \ref{lem:caracterizacion_existencia_coextuniv}(a,c).
For the proof of $(a)\Rightarrow(b)$ recall that, for every abelian
group $G$ and every positive integer $n$, we have that $G/nG\cong\Ext[\mathbb{Z}(n)][1][\Ab][G]$
(see \cite[p.267]{fuchs2015abelian}). Therefore, if $T$ is co-$\operatorname{Ext}^{1}$-universal
in $\mathcal{T}$ and $G=T^{X}/\overline{t}(T^{X})$ for a set $X$,
then it follows from Lemma \ref{lem:caracterizacion_existencia_coextuniv}(a,c)
that $nG=G$ for all $n>0$. And thus, $G$ is injective (see \cite[Chapter 4, Theorem 2.6]{fuchs2015abelian}).
It remains to prove $(a),(b)\Rightarrow(c)$. Let $Q$ be a quotient
of $T$ and $T$ be co-$\operatorname{Ext}^{1}$-universal in $\mathcal{T}_{Z}$.
Observe that, for every set $X$, $Q^{X}/\overline{t}(Q^{X})$ is
a quotient of $T^{X}/\overline{t}(T^{X})$. Then, it follows from
(b) that $Q^{X}/\overline{t}(Q^{X})$ injective since quotients of
injective groups are also injective in $\Ab$. And hence, $Q$ is
co-$\operatorname{Ext}^{1}$-universal in $\mathcal{T}$.
\end{proof}
\begin{cor}
\label{exa:contraejemplo }Let $p>0$ be a prime number. Then, $\bigoplus_{n\in\mathbb{N}}\mathbb{Z}(p^{n})$
is not co-$\operatorname{Ext}^{1}$-universal in $\mathcal{T}_{Z}$ (resp.
$\mathcal{T}_{p}$). 
\end{cor}

\begin{proof}
Consider $R:=\bigoplus_{n\in\mathbb{N}}\mathbb{Z}(p^{n})$. By Proposition
\ref{prop:VX cociente torsion es divisible}(a,b), it is enough to
prove that $R^{\mathbb{N}}/t(R^{\mathbb{N}})$ is not injective. For
this, observe that $P:=\prod_{n\in\mathbb{N}}\mathbb{Z}(p^{n})$ is
a quotient of $R^{\mathbb{N}}$. And thus, $P/t(P)$ is a quotient
of $R^{\mathbb{N}}/t(R^{\mathbb{N}})$. Therefore, we only need to
prove that $P/t(P)$ is not divisible. Let $x:=(x_{n})_{n\in\mathbb{N}}\in\prod_{n\in\mathbb{N}}\mathbb{Z}(p^{n})$
with $x_{n}=1+p^{n}\mathbb{Z}$, for all $n\in\mathbb{N}$. If we assume
that there is $\alpha:=(\alpha_{n})_{n\geq1}\in\prod_{n\in\mathbb{N}}\mathbb{Z}(p^{n})$
such that $x+t(P)=p\alpha+t(P)$, then there is $m>1$ such that $p^{m}(x-p\alpha)=0$,
and thus, $p^{n}$ divides $p^{m}(1-p\alpha_{n})$, for all $n\geq1$.
This cannot be satisfied for $n>m$. Therefore, $P$ is not divisible. 
\end{proof}
\begin{fact}
\label{fact:basic}Let $p$ be a prime number. Recall that every $A\in\mathcal{T}_{p}$
has a \textbf{$p$-basic subgroup} $B$ (see \cite[Chapter 5, Theorem 5.2]{fuchs2015abelian}).
That is, (i) $B$ is a coproduct of cyclic groups, (ii) $p^{k}B=B\cap p^{k}A$,
for all $k>0$, and (iii) $p\left(A/B\right)=A/B$. We have the following
properties for a $p$-basic subgroup $B$ of a group $A\in\mathcal{T}_{p}$. 
\begin{enumerate}
\item If $B$ is bounded and $A$ is reduced, then $A=B$. Indeed, if $B$
is bounded, then it follows from (ii) that $B$ is a direct summand
of $A$ (see \cite[Chapter 5, Theorem 2.5]{fuchs2015abelian}). And
hence, since $A$ is reduced, it follows from (iii) that $A=B$. 
\item In particular, if $A$ is reduced and not bounded, then $B$ is
not bounded. 
\item Observe that, if $B$ is not bounded, then $B$ has a quotient isomorphic
to $\bigoplus_{n\in\mathbb{N}}\mathbb{Z}(p^{n})$. 
\end{enumerate}
\end{fact}

\begin{thm}
\label{thm:main}The following statements are equivalent for $T\in\mathcal{T}_{Z}$: 
\begin{enumerate}
\item $T$ is co-$\operatorname{Ext}^{1}$-universal in $\mathcal{T}_{Z}$; 
\item $t_{p}(T)$ is co-$\operatorname{Ext}^{1}$-universal in $\mathcal{T}_{p}$,
for every prime number $p$;
\item $t_{p}(T)=D_{p}\oplus R_{p}$, where $D_{p}$ is injective and $R_{p}$
is reduced and bounded, for every prime number $p$.
\end{enumerate}
\end{thm}

\begin{proof}
Let $\mathcal{P}$ be the set of prime numbers. We put $T_{p}:=t_{p}(T)$,
for every $p\in\mathcal{P}$; and note that $T=\bigoplus_{p\in\mathcal{P}}T_{p}$.

$(a)\Leftrightarrow(b)$ Let $X$ be a set. Observe that $T^{X}/t(T^{X})\cong\bigoplus_{p\in\mathcal{P}}T_{p}^{X}/t(T_{p}^{X})$.
And hence, $T^{X}/t(T^{X})$ is injective if and only if $T_{p}^{X}/t(T_{p}^{X})$
is injective for all $p\in\mathcal{P}$ (see \cite[Chapter 4, Theorem 3.1]{fuchs2015abelian}).
Therefore, the equivalence $(a)\Leftrightarrow(b)$ follows from Proposition
\ref{prop:VX cociente torsion es divisible}(a,b). 

$(b)\Rightarrow(c)$ Let $p\in\mathcal{P}$ and $T_{p}=D_{p}\oplus R_{p}$,
where $D_{p}$ is injective and $R_{p}$ is reduced. If $R_{p}$ is
not bounded, then it follows from \cite[Chapter 5, Theorem 6.10]{fuchs2015abelian}
and Fact \ref{fact:basic}(c) that there is a quotient of $R_{p}$
isomorphic to $\bigoplus_{n\in\mathbb{N}}\mathbb{Z}(p^{n})$. And
hence, by Proposition \ref{prop:VX cociente torsion es divisible}(a,c),
we get that $\bigoplus_{n\in\mathbb{N}}\mathbb{Z}(p^{n})$ is co-$\operatorname{Ext}^{1}$-universal
in $\mathcal{T}$, which contradicts Corollary \ref{exa:contraejemplo }. 

$(c)\Rightarrow(b)$ Let $p\in\mathcal{P}$ and $T_{p}=D_{p}\oplus R_{p}$,
where $D_{p}$ is injective and $R_{p}$ is reduced and bounded. By
Corollary \ref{cor:coproducto de universales} and Proposition \ref{prop:VX cociente torsion es divisible}(a,b),
it is enough to prove that $R_{p}$ is co-$\operatorname{Ext}^{1}$-universal
in $\mathcal{T}_{Z}$. And this follows from Proposition \ref{prop:VX cociente torsion es divisible}(a,b)
since $R_{p}^{X}=t(R_{p}^{X})$ for every set $X$. 
\end{proof}

\section*{Acknowledgments}

The research presented in this article began in December 2021, when
the first named author was a postdoctoral fellow at the Instituto
de Matem\'atica y Estad\'istica Rafael Laguardia. The research project
was continued and completed when the first author started a postdoctoral
stay at Centro de Ciencias Matem\'aticas, UNAM Campus Morelia, in
April 2022. The first author would like to thank all the academic
and administrative staff of these institutions for their warm hospitality,
and in particular Raymundo Bautista (CCM, UNAM), Marcelo Lanzilotta
(IMERL) and Marco A. P\'erez (IMERL) for all their support.

\bibliographystyle{abbrv}
\bibliography{Bibliografia}
\end{document}